\documentclass[11pt,a4paper,reqno]{amsart}

\usepackage{times}
\usepackage{amsmath}
\usepackage{amssymb}
\usepackage{amsthm}
\usepackage{latexsym}
\usepackage{amsfonts}
\usepackage{xcolor}
\usepackage{graphicx}
\DeclareGraphicsExtensions{.eps,.pdf,.png,.jpg}
\usepackage{bbm}
\usepackage{tikz-cd}
\usepackage{adjustbox}
\usepackage[colorlinks=true, citecolor=blue, linkcolor=blue, anchorcolor=black]{hyperref}
\usepackage{cases}

\newtheorem{thm}{Theorem}[section]
\newtheorem{cor}[thm]{Corollary}
\newtheorem{lem}[thm]{Lemma}
\newtheorem{prop}[thm]{Proposition}
\theoremstyle{definition}

\newtheorem{rem}[thm]{Remark}

\numberwithin{equation}{section}
\allowdisplaybreaks

\newcommand{\innerprod}[2]{\left\langle #1,\, #2 \right\rangle} 

\newcommand{\one}{\mathbbm{1}}
\newcommand{\braces}[1]{{\rm (}#1{\rm )}}
\newcommand{\rmref}[1]{{\rm\ref{#1}}}

\renewcommand{\Re}{\operatorname{Re}}
\newcommand{\linspan}{\operatorname{span}}
\newcommand{\ran}{\operatorname{ran}}
\newcommand{\dom}{\operatorname{dom}}

\newcommand{\wt}{\widetilde}
\newcommand{\wh}{\widehat}
\newcommand{\<}{\langle}
\renewcommand{\>}{\rangle}

\newcommand{\R}{\ensuremath{\mathbb R}}    
\newcommand{\C}{\ensuremath{\mathbb C}}    
\newcommand{\N}{\ensuremath{\mathbb N}}    
\newcommand{\Z}{\ensuremath{\mathbb Z}}    

\newcommand{\calA}{\mathcal A}         
\newcommand{\calB}{\mathcal B}         \newcommand{\frakB}{\mathfrak B}

\newcommand{\calE}{\mathcal E}

\newcommand{\calH}{\mathcal H}         
\newcommand{\calI}{\mathcal I}         
         
\newcommand{\calK}{\mathcal K}         
\newcommand{\calL}{\mathcal L}

\newcommand{\calX}{\mathcal X}

				 \newcommand{\bE}{\mathbb E}

				 \newcommand{\bH}{\mathbb H}

				 \newcommand{\bP}{\mathbb P}

\newcommand{\la}{\lambda}
\newcommand{\veps}{\varepsilon}

\newcommand{\vphi}{\varphi}

\newcommand{\Tr}{\operatorname{Tr}}
\newcommand{\rank}{\operatorname{rank}}
\newcommand{\dist}{\operatorname{dist}}
\newcommand{\supp}{\operatorname{supp}}
\newcommand{\diag}{\operatorname{diag}}

\title[Error bounds for kernel-based approximations of the Koopman operator]{Error bounds for kernel-based approximations of the Koopman operator}

\begin{document}
\begin{abstract}
We consider the data-driven approximation of the Koopman operator for stochastic differential equations on reproducing kernel Hilbert spaces (RKHS). Our focus is on the estimation error if the data are collected from long-term ergodic simulations. We derive both an exact expression for the variance of the kernel cross-covariance operator, measured in the Hilbert-Schmidt norm, and probabilistic bounds for the finite-data estimation error. Moreover, we derive a bound on the prediction error of observables in the RKHS using a finite Mercer series expansion. Further, assuming Koopman-invariance of the RKHS, we provide bounds on the full approximation error. Numerical experiments using the Ornstein-Uhlenbeck process illustrate our results.
\end{abstract}

\author[F.~Philipp]{Friedrich Philipp}
\address{{\bf F.~Philipp}
	Technische Universit\"at Ilmenau, Institute for Mathematics,
	Weimarer Stra\ss e 25, D-98693 Ilmenau, Germany}
\email{friedrich.philipp@tu-ilmenau.de}

\author[M.~Schaller]{Manuel Schaller}
\address{{\bf M.~Schaller}
	Technische Universit\"at Ilmenau, Institute for Mathematics,
	Weimarer Stra\ss e 25, D-98693 Ilmenau, Germany}
\email{manuel.schaller@tu-ilmenau.de}

\author[K.~Worthmann]{Karl Worthmann}
\address{{\bf K.~Worthmann}
	Technische Universit\"at Ilmenau, Institute for Mathematics,
	Weimarer Stra\ss e 25, D-98693 Ilmenau, Germany}
\email{karl.worthmann@tu-ilmenau.de}

\author[S.~Peitz]{Sebastian Peitz}
\address{{\bf S.~Peitz}
	Paderborn University, Department of Computer Science, Data Science for Engineering, Germany}
\email{sebastian.peitz@upb.de}

\author[F.~N\"uske]{Feliks N\"uske}
\address{{\bf F.~N\"uske}
	Max Planck Institute for Dynamics of Complex Technical Systems, Magdeburg, Germany}
\email{nueske@mpi-magdeburg.mpg.de}

\maketitle

\section{Introduction}
The Koopman operator~\cite{Koo31} has become an essential tool in the modeling process of complex dynamical systems based on simulation or measurement data. The philosophy of the Koopman approach is that for a (usually non-linear) dynamical system on a finite-dimensional 
space, the time-evolution of expectation values of observable functions 
satisfies a linear differential equation. 
Hence, after ``lifting'' the dynamical system into an infinite-dimensional function space of observables, linear methods become available for its analysis. The second step is then to notice that traditional Galerkin approximations of the Koopman operator can be consistently estimated from simulation or measurement data, establishing the fundamental connection between the Koopman approach and modern data science. 
Koopman methods have found widespread application in system identification~\cite{brunton16}, control~\cite{KordMezi18,peitz_klus19,KordMezi20,kaiser21,SchaWort22}, sensor placement~\cite{manohar18}, molecular dynamics~\cite{schuette99,prinz2011,noe13,nueske14,klus16,wu20}, and many other fields. We refer to \cite{klus18_rev,mauroy20,brunton22} for comprehensive reviews of the state of the art.

The fundamental numerical method for the Koopman approach is \emph{Extended Dynamic Mode Decomposition} (EDMD)~\cite{williams15}, which allows to learn a Galerkin approximation of the Koopman operator from finite (simulation or measurement) data on a subspace spanned by a finite set of observables, 
often called 
dictionary. 
An appropriate choice of 
said dictionary is a 
challenging problem. 
In light of this issue, representations of the Koopman operator on large approximation spaces have been considered in recent years, including deep neural networks~\cite{lusch18,mardt18}, tensor product spaces~\cite{klus16_tensor,nueske21}, and \emph{reproducing kernel Hilbert spaces} (RKHS)~\cite{williams15_kernel,gia19,klus20}. See also~\cite{bblshh,das2021,gumah22,gab,gnaab,kostic22,kostic23} for recent studies on the use of reproducing kernels in the context of dynamical systems. In Reference~\cite{klus20} it was shown that by means of the integral operator associated to an RKHS, it is possible to construct a type of Galerkin approximation of the Koopman operator. 
The central object are (cross-)covariance operators, which can be 
estimated from 
data, using only evaluations of the feature map. Due to the relative simplicity of the resulting numerical algorithms on the one hand, and the rich approximation properties of reproducing kernels on the other hand, kernel methods have emerged as a promising candidate to overcome the fundamental problem of dictionary selection.

A key 
question 
is the quantification of the estimation error for (compressed\footnote{A compression of a linear operator $T$ to a subspace $M$ is given by $PT|_M$, where $P$ denotes a projection onto $M$.}) Koopman operators. For finite dictionaries and independent, identically distributed (i.i.d.) samples, error estimates were provided in~\cite{kurdila18,nueske23}, see also \cite{zuazua21} for the ODE case and~\cite{SchaWort22} for an extension to control-affine systems. The estimation error for cross-covariance operators on kernel spaces was considered in~\cite{mollenhauer22}, where general concentration inequalities were employed. The data were also allowed to be correlated, and mixing coefficients were used to account for the lack of independence.
In this article, we take a different route and follow the approach of our previous paper~\cite{nueske23}, where we, in addition, also derived error estimates for the Koopman generator and operator 
for finite dictionaries and data collected from long-term, ergodic trajectories. 
This setting is relevant in many areas of science, where sampling i.i.d.\ from an unknown stationary distribution
is practically infeasible, e.g., in fluid or molecular dynamics. The centerpiece of our results was an exact expression for the variance of the finite-data estimator, which can be bounded by an asymptotic variance. The asymptotic variance by itself is a highly interesting dynamical quantity, which can also be described in terms of Poisson equations for the generator \cite[Section 3]{lelievre16}.

We consider the Koopman semigroup $(K^t)_{t\ge 0}$ generated by a stochastic differential equation on the space $L^2_\mu$, where $\mu$ is a probability measure which is invariant w.r.t.\ the associated Markov process. We study the action of $K^t$ on observables in an RKHS $\bH$ which is densely and compactly embedded in $L^2_\mu$. If this action is considered through the ``lens'' of the kernel integral operator $\calE : L^2_\mu \to \bH$ (see Section~\ref{subsec:rkhs}), we arrive at a family of operators $C_\bH^t = \calE K^t\calE^*$ (cf.\ Figure~\ref{fig:diag1}). The action of $C_\bH^t : \bH\to\bH$ is that of a cross-covariance operator:
$$
C_\bH^t\psi = \int (K^t\psi)(x)k(x,\cdot)\,d\mu(x),\qquad\psi\in\bH,
$$
where $k(\cdot,\cdot)$ is the kernel generating the RKHS $\bH$. These operators possess canonical empirical estimators based on finite simulation data, which only require evaluations of the feature map.

\begin{figure}[ht!]
\centering
\adjustbox{scale=1.3,center}{ 
\begin{tikzcd}[column sep = large]
L^2_\mu \arrow{r}{K^t}  & L^2_\mu\arrow{d}{\calE} \\
\bH \arrow{u}{\calE^*} \arrow{r}{C^t_{\bH}}& \bH 
\end{tikzcd}}
\caption{Diagram illustrating the different operators involved}
\label{fig:diag1}
\end{figure}

Our contribution, illustrated in Figure \ref{fig:schaubild}, is two-fold. In our first main result, Theorem \ref{t:Eest}, we provide an exact formula for the Hilbert-Schmidt variance of the canonical empirical estimator $\wh C_\bH^{m, t}$ of the cross-covariance operator $C_\bH^t$, for $m$ data points sampled from a long ergodic simulation. This result holds under the very mild assumption that $\la=1$ is a simple\footnote{Simplicity of $\la=1$ as an eigenvalue of $K^t$ is already guaranteed by ergodicity, cf.\ \ref{a:ergodicity}. Hence, the essential assumption here is that the eigenvalue is isolated.} isolated eigenvalue of $K^t$, which does not exclude deterministic systems, extends the findings in~\cite{nueske23} to the kernel setting and no longer depends on the dictionary size (which would be infinite, at any rate). Furthermore, the result allows for probabilistic estimates for the error $\|\wh C_\bH^{m,t} - C_\bH^t\|_{HS}$, see Proposition \ref{p:prob_est}. 

As a second main result, we propose an empirical estimator for the restriction of the Koopman operator $K^t$ to $\bH$, truncated to finitely many terms of its estimated Mercer series expansion, and prove a probabilistic bound for the resulting estimation error in Theorem \ref{t:main}, measured in the operator norm for bounded linear maps from $\bH$ to $L^2_\mu$. This result can be seen as a bound on the prediction error for the RKHS-based Koopman operator due to the use of finite data. In the situation where the RKHS is invariant under the Koopman operator we are able to complement the preceding error analysis with a bound on the full approximation error in Theorem \ref{t:proj_err}.

Finally, we illustrate our results for a one-dimensional Ornstein-Uhlenbeck (OU) process. For this simple test case, all quantities appearing in our error estimates are known analytically and can be well approximated numerically. Therefore, we are able to provide a detailed comparison between the error bound obtained from our results and the actual errors observed for finite data. Our experiments show that our bounds for the estimation error of the cross-covariance operator are accurate, and that the corrections we introduced to account for the inter-dependence of the data are indeed required. Concerning the prediction error, we find our theoretical bounds still far too conservative, which reflects the problem of accounting for the effect of inverting the mass matrix in traditional EDMD. This finding indicates that additional research is required on this end.

\begin{figure}[ht!]
\centering
\includegraphics[width=\columnwidth]{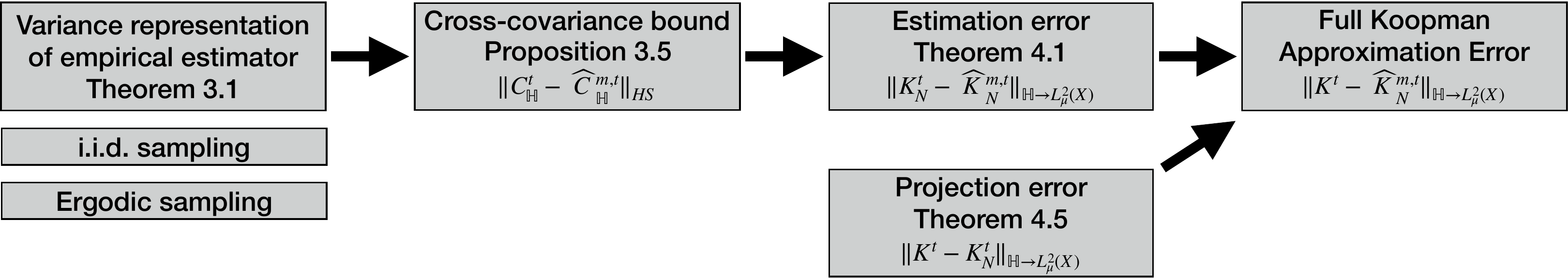}
\caption{Illustration of main results}
\label{fig:schaubild}
\end{figure}

The paper is structured as follows: the setting is introduced in Section~\ref{sec:prelim}. The result concerning the variance of the empirical cross-covariance operator, Theorem \ref{t:Eest}, is presented and proved in Section~\ref{sec:res1}, while our bound for the prediction error is part of Theorem \ref{t:main} in Section~\ref{sec:res2}. Numerical experiments are shown in Section~\ref{sec:numerics}, conclusions are drawn in Section~\ref{sec:conclusions}.

\section{Preliminaries}\label{sec:prelim}
In this section, we review the required background on stochastic differential equations (Section~\ref{subsec:sde}), reproducing kernel Hilbert spaces (Section~\ref{subsec:rkhs}), Koopman operators (Section~\ref{subsec:koopman}), their representations on an RKHS (Section~\ref{subsec:kernel_koop}),and the associated empirical estimators (Section~\ref{subsec:emp}). The results in this section can all be found in the literature, but we list them here at any rate to achieve a self-contained presentation. Selected proofs are also shown in the appendix for the reader's convenience. Below, we list the most frequently used notation:

\begin{center}
\begin{tabular}{| c | l |}
\hline
\textbf{Symbol} & \textbf{Meaning} \\
\hline
$b, \sigma$ & Drift and diffusion of SDE \\
$\calX$ & State space \\
$\mu$ & Invariant measure \\
$L^p_\mu(\calX)$ & $\mu$-weighted Lebesgue space on $\calX$ \\
$k$ & Kernel function \\
$\Phi$ & Kernel feature map \\
$\varphi$ & Diagonal of the kernel \\
$K^t$ & Koopman operator \\
$\calL$ & Infinitesimal generator \\
$\mu_j, \psi_j$ & Eigenvalues and eigenfunctions of $\calL$ \\
$C^t_\bH$ & Time-lagged kernel cross-covariance operator \\
$C_\bH$ & Kernel covariance operator\\
\hline
\end{tabular}
\end{center}

\subsection{Stochastic differential equations}\label{subsec:sde}
Let $\calX\subset\R^d$ and let a stochastic differential equation (SDE) with drift vector field $b : \calX\to\R^d$ and diffusion matrix field $\sigma : \calX\to\R^{d\times d}$ be given, i.e.,
\begin{equation}\label{e:SDE}
dX_t = b(X_t)\,dt + \sigma(X_t)\,dW_t,
\end{equation}
where $W_t$ is $d$-dimensional Brownian motion. We assume that both $b$ and $\sigma$ are Lipschitz-continuous and that $(1+\|\cdot\|_2)^{-1}[\|b\|_2 + \|\sigma\|_F]$ is bounded on $\calX$. Then \cite[Theorem 5.2.1]{oe} guarantees the existence of a unique solution $(X_t)_{t\ge 0}$ to \eqref{e:SDE}.

The solution $(X_t)_{t\ge 0}$ constitutes a continuous-time Markov process whose transition kernel will be denoted by $\rho_t : \calX\times \calB_\calX\to\R$, where $\calB_\calX$ denotes the Borel $\sigma$-algebra on $\calX$. Then $\rho_t(x,\cdot)$ is a probability measure for all $x\in \calX$, and for each $A\in\calB_\calX$ we have that $\rho_t(\cdot,A)$ is a representative of the conditional probability for $A$ containing $X_t$ given $X_0=\,\cdot\,$, i.e.,
$$
\rho_t(x,A) = \bP(X_t\in A|X_0=x)\quad\text{for $\bP_{X_0}$-a.e. $x\in \calX$},
$$
where $\bP_{X_0}$ denotes the marginal distribution of $X_0$.

Throughout, we will assume the existence of an {\em invariant} (Borel) {\em probability measure} $\mu$ for the Markov process $(X_t)_{t\ge 0}$, i.e., we have
\begin{equation}\label{e:set_inv}
\int\rho_t(x,A)\,d\mu(x) = \mu(A)
\end{equation}
for all $t\ge 0$.


In addition to invariance, we assume that $\mu$ is \emph{ergodic}, meaning that for any $t > 0$ every $\rho_t$-invariant set $A$ (that is, $\rho_t(x, A) = 1$ for all $x\in A$) satisfies $\mu(A)\in\{0,1\}$. In this case, the Birkhoff ergodic theorem~\cite[Theorem 9.6]{kall} (see also \eqref{e:birkhoff}) and its generalizations apply, and allow us to calculate expectations w.r.t. $\mu$ using long-time averages over simulation data.

We let $\|\cdot\|_p$ denote the $L_\mu^p(\calX)$-norm, $1\le p < \infty$. In the particular case $p=2$, scalar product and norm on the Hilbert space $L_\mu^2(\calX)$ will be denoted by $\<\cdot\,,\cdot\>_\mu$ and $\|\cdot\|_\mu$, respectively.

\subsection{Reproducing kernel Hilbert spaces}
\label{subsec:rkhs}
In what follows, let $k : \calX\times \calX\to\R$ be a continuous and symmetric positive definite kernel, that is, we have $k(x,y) = k(y,x)$ for all $x,y\in \calX$ and
$$
\sum_{i,j=1}^m k(x_i,x_j)c_ic_j\,\ge\,0
$$
for all choices of $x_1,\ldots,x_m\in \calX$ and $c_1,\ldots,c_m\in\R$. It is well known that $k$ generates a so-called {\em reproducing kernel Hilbert space} (RKHS) \cite{a,bt,pr} $(\bH,\<\cdot\,,\cdot\>)$ of continuous functions, such that for $\psi \in \bH$ the {\em reproducing property}
\begin{equation}\label{e:reproducing_property}
\psi(x) = \<\psi,\Phi(x)\>,\qquad x\in \calX,
\end{equation}
holds, where $\Phi : \calX\to\bH$ denotes the so-called {\em feature map} corresponding to the kernel $k$, i.e.,
$$
\Phi(x) = k(x,\cdot),\qquad x\in \calX.
$$
In the sequel, we shall denote the norm on $\bH$ by $\|\cdot\|$ and the kernel diagonal by $\vphi$:
$$
\vphi(x) = k(x,x),\qquad x\in \calX.
$$
Then for $x\in \calX$ we have
$$
\|\Phi(x)\|^2 = \<\Phi(x),\Phi(x)\> = \<k(x,\cdot),k(x,\cdot)\> = k(x,x) = \vphi(x).
$$
We shall frequently make use of the following estimate:
$$
|k(x,y)| = |\<\Phi(x),\Phi(y)\>|\le\|\Phi(x)\|\|\Phi(y)\| = \sqrt{\vphi(x)\vphi(y)}.
$$
In particular, it shows that $k$ is bounded if and only if its diagonal $\vphi$ is bounded.

By $\calL_\mu^p(\calX)$, $p\in [1,\infty)$, we denote the space of all {\em functions} (not equivalence classes) on $\calX$ with a finite $p$-norm $\|\cdot\|_p$. Henceforth, we shall impose the following

\medskip\noindent
{\bf Compatibility Assumptions:}
\begin{enumerate}
\item[(A1)] $\vphi\in\calL^2_\mu(\calX)$.
\item[(A2)] If $\psi\in L^2_\mu(\calX)$ such that $\int\!\int k(x,y)\psi(x)\psi(y)\,d\mu(x)\,d\mu(y) = 0$, then $\psi=0$.
\item[(A3)] If $\psi\in\bH$ such that $\psi(x)=0$ for $\mu$-a.e.\ $x\in \calX$, then $\psi(x) = 0$ for all $x\in \calX$.
\end{enumerate}

Many of the statements in this subsection can also be found in \cite[Chapter 4]{sc}. However, as we aim to present the contents in a self-contained way, we provide the proofs in \ref{a:proofs}.

The following lemma explains the meaning of the compatibility assumptions (A1) and (A2), cf.\ \cite[Theorem 4.26]{sc}.

\begin{lem}\label{l:basics0}
Under the assumption that $\vphi\in\calL^1_\mu(\calX)$ \braces{in particular, under assumption {\rm (A1)}}, we have that $\bH\subset \calL^2_\mu(\calX)$ with
\begin{equation}\label{e:bounded_embedding}
\|\psi\|_\mu\,\le\,\sqrt{\|\vphi\|_1}\cdot\|\psi\|,\qquad\psi\in\bH,
\end{equation}
and assumption {\rm (A2)} is equivalent to the density of $\bH$ in $\calL^2_\mu(\calX)$.
\end{lem}

We have meticulously distinguished between functions and equivalence classes as there might be distinct functions $\phi$ and $\psi$ in $\bH$, which are equal $\mu$-almost everywhere\footnote{For example, if $\mu = \delta_a$ and $\phi(a) = \psi(a)$}, i.e., $\phi=\psi$ in $L^2_\mu(\calX)$. The compatibility assumption (A3) prohibits this situation so that $\bH$ can in fact be seen as a subspace of $L^2_\mu(\calX)$, which is then densely and continuously embedded.

\begin{rem}
(a) Condition (A1) implies $k\in L^4_{\mu\otimes\mu}(\calX\times \calX)$, where $\mu \otimes \mu$ is the product measure on $\calX\times \calX$.

\smallskip
\noindent(b) The density of $\bH$ in $L^2_\mu(\calX)$ is strongly related to the term {\em universality} in the literature, see \cite{sfl}.

\smallskip
\noindent(c) Condition (A3) holds if $\supp\mu = \calX$, cf.\ \cite[Exercise 4.6]{sc}.
\end{rem}

It immediately follows from
\begin{align}\label{e:calE_bounded}
\int|\psi(x)|\|\Phi(x)\|\,d\mu(x)\le\|\psi\|_\mu\|\vphi\|_1^{1/2},
\end{align}
for $\psi\in L^2_\mu(\calX)$ that the linear operator $\calE : L^2_\mu(\calX)\to\bH$, defined by
$$
\calE\psi := \int\psi(x)\Phi(x)\,d\mu(x),\qquad\psi\in L^2_\mu(\calX),
$$
is well defined (as a Bochner integral in $\bH$) and bounded with operator norm not larger than~$\|\vphi\|_1^{1/2}$.

\begin{rem}
The so-called {\em kernel mean embedding} $\calE_k$, mapping probability measures $\nu$ on $\calX$ to the RKHS $\bH$, is defined by $\calE_k\nu = \int\Phi(x)\,d\nu(x)$, see, e.g., \cite{sgss}. Hence, we have $\calE\psi = \calE_k\nu$ with $d\nu = \psi\,d\mu$.
\end{rem}

Note that the operator $\calE$ is not an embedding in strict mathematical terms. The terminology {\em embedding} rather applies to its adjoint $\calE^*$. Indeed, the operator $\calE$ enjoys the simple but important property:
\begin{align}\label{e:adj}
\<\calE\psi,\eta\> = \int\psi(x)\<\Phi(x),\eta\>\,d\mu(x) = \int\psi(x)\eta(x)\,d\mu(x) = \<\psi,\eta\>_\mu
\end{align}
for $\psi\in L^2_\mu(\calX)$ and $\eta\in\bH$. This implies that the adjoint operator $\calE^* : \bH\to L^2_\mu(\calX)$ is the inclusion operator from $\bH$ into $L^2_\mu(\calX)$, i.e.,
\begin{align}\label{e:embedding}
\calE^*\eta = \eta,\qquad \eta\in\bH.
\end{align}
We shall further define the covariance operator\footnote{In what follows, by $L(\calH,\calK)$ we denote the set of all bounded (i.e., continuous) linear operators between Hilbert spaces $\calH$ and $\calK$. As usual, we also set $L(\calH) := L(\calH,\calH)$.}
$$
C_\bH := \calE\calE^*\,\in\,L(\bH).
$$
Recall that a linear operator $T\in L(\calH)$ on a Hilbert space $\calH$ is {\em trace class} if for some (and hence for each) orthonormal basis $(e_j)_{j\in\N}$ of $\calH$ we have that $\sum_{j=1}^\infty\<(T^*T)^{1/2}e_i,e_i\> < \infty$. A linear operator $S\in L(\calH,\calK)$ between Hilbert spaces $\calH$ and $\calK$ is said to be {\em Hilbert-Schmidt} \cite[Chapter III.9]{gk} if $S^*S$ is trace class, i.e., $\|S\|_{HS}^2 := \sum_{j=1}^\infty\|Se_i\|^2 < \infty$ for some (and hence for each) orthonormal basis $(e_j)_{j\in\N}$.

\begin{lem}[{\cite[Theorem 4.27]{sc}}]\label{l:ECH}
Let the Compatibility Assumptions {\rm (A1)--(A3)} be satisfied. Then the following hold.
\begin{enumerate}
\item[{\rm (a)}] The operator $\calE$ is an injective Hilbert-Schmidt operator with
$$
\|\calE\|_{HS}^2 = \|\vphi\|_1.
$$
\item[{\rm (b)}] The space $\bH$ is densely and compactly embedded in $L^2_\mu(\calX)$.
\item[{\rm (c)}] The operator $C_\bH$ is an injective non-negative self-adjoint trace class operator.
\end{enumerate} 
\end{lem}

The next theorem is due to Mercer and can be found in, e.g., \cite{rn}. It shows the existence of a particular orthonormal basis $(e_j)_{j=1}^\infty$ of $L^2_\mu(\calX)$ composed of eigenfunctions of $\calE^*\calE$, which we shall henceforth call the {\em Mercer basis} corresponding to the kernel $k$. Again for the sake of self-containedness, we give a short proof in \ref{a:proofs}.

\begin{thm}[Mercer's Theorem]\label{t:mercer}
There exists an orthonormal basis $(e_j)_{j=1}^\infty$ of $L^2_\mu(\calX)$ consisting of eigenfunctions of $\calE^*\calE$ with corresponding eigenvalues $\la_j > 0$ such that $\sum_{j=1}^\infty\la_j = \|\vphi\|_1 < \infty$. Furthermore, $(f_j)_{j=1}^\infty$ with $f_j = \sqrt{\la_j}e_j$ constitutes an orthonormal basis of $\bH$ consisting of eigenfunctions of $C_\bH$ with corresponding eigenvalues $\la_j$. Moreover, for all $x,y\in \calX$,
$$
k(x,y) = \sum_jf_j(x)f_j(y) = \sum_j\la_je_j(x)e_j(y),
$$
the series converges absolutely.
\end{thm}

\subsection{The Koopman semigroup}
\label{subsec:koopman}
The {\em Koopman semigroup} $(K^t)_{t\ge 0}$ associated with the SDE~\eqref{e:SDE} is defined by
$$
(K^t\psi)(x) = \bE[\psi(X_t)|X_0=x] = \int\psi(y)\,\rho_t(x,dy),
$$
for $\psi\in B(\calX)$, the set of all bounded Borel-measurable functions on $\calX$, and 
$\rho_t(x,dy) = d\rho_t(x,\cdot)(y)$. It is easy to see that the invariance of $\mu$ is equivalent to the identity
\begin{equation}\label{e:invariance}
\int K^t\psi\,d\mu = \int\psi\,d\mu
\end{equation}
for all $t\ge 0$ and $\psi\in B(\calX)$ (which easily extends to functions $\psi\in L^1_\mu(\calX)$, see Proposition \ref{p:Kbounded}).

\begin{rem}
Note that in the case $\sigma=0$ the SDE \eqref{e:SDE} reduces to the deterministic ODE $\dot x = b(x)$. Then \eqref{e:invariance} implies $\int|\psi(\phi(t,x))|^2\,d\mu(x) = \int|\psi(x)|^2\,d\mu(x)$ for all $t\ge 0$ and all $\psi\in B(\calX)$, where $\phi(\cdot,x)$ is the solution of the initial value problem $\dot y = b(y)$, $y(0)=x$. Hence, the composition operator $K^t : \psi\mapsto\psi\circ\phi(t,\cdot)$ is unitary in $L^2_\mu(\calX)$.
\end{rem}

The proofs of the following two propositions can be found in \ref{a:proofs}.

\begin{prop}\label{p:Kbounded}
For each $p\in [1,\infty]$ and $t\ge 0$, $K^t$ extends uniquely to a bounded operator from $L^p_\mu(\calX)$ to itself with operator norm $\|K^t\|_{L^p_\mu\to L^p_\mu}\le 1$.
\end{prop}

By $C_b(\calX)$ we denote the set of all bounded continuous functions on $\calX$. As the measure $\mu$ is finite, we have $C_b(\calX)\subset B(\calX)\subset L^p_\mu(\calX)$ for all $p\in [1,\infty]$. In fact, $C_b(\calX)$ is dense in each $L^p_\mu(\calX)$, $p\in [1,\infty)$, see \cite[Theorem 3.14]{rrca}.

\begin{prop}\label{p:semigroup}
$(K^t)_{t\ge 0}$ is a $C_0$-semigroup of contractions in $L^p_\mu(\calX)$ for each $p\in [1,\infty)$.
\end{prop}


The {\em infinitesimal generator} of the $C_0$-semigroup $(K^t)_{t\ge 0}$ is the (in general unbounded) operator in $L^2_\mu(\calX)$, defined by
\begin{align}\label{e:generator}
\calL\psi = L^2_\mu\text{-}\lim_{t\to 0}\frac{K^t\psi - \psi}t,
\end{align}
whose domain $\dom\calL$ is the set of all $\psi\in L^2_\mu(\calX)$ for which the above limit exists. By Proposition~\ref{p:semigroup} and the Lumer-Phillips theorem (see \cite{lp}), the operator $\calL$ is densely defined, closed\footnote{Recall that a linear operator $T$, defined on a subspace $\dom T$ of a Hilbert space $\calH$, which maps to a Hilbert space $\calK$, is closed if its graph is closed in $\calH\times\calK$.}, dissipative (i.e., $\Re\<\calL\psi,\psi\>_\mu\le 0$ for all $\psi\in\dom\calL$), and its spectrum is contained in the closed left half-plane.

\begin{lem}\label{l:one}
The constant function $\one$ is contained in $\dom\calL$ and $\calL\one = 0$. Moreover, both $M := \linspan\{\one\}$ and $M^\perp$ are invariant under $\calL$ and all $K^t$, $t\ge 0$.
\end{lem}
\begin{proof}
It is easy to see that $K^t\one = \one$ for each $t\ge 0$ and hence $\one\in\dom\calL$ with $\calL\one = 0$. Hence $K^tM\subset M$ for all $t\ge 0$ and $\calL M\subset M$. Now, if $\psi\in M^\perp$, then $\<K^t\psi,\one\>_\mu = \int K^t\psi\,d\mu = \int\psi\,d\mu = \<\psi,\one\>_\mu = 0$, which shows that also $K^tM^\perp\subset M^\perp$. The relation $\calL M^\perp\subset M^\perp$ follows from \eqref{e:generator}.
\end{proof}

\subsection{Representation of Koopman Operators on the RKHS}
\label{subsec:kernel_koop}
Using the integral operator $\mathcal{E}$, it is possible to represent the Koopman operator with the aid of a linear operator on $\bH$, which is based on kernel evaluations. This construction mimics the well-known kernel trick used frequently in machine learning. To begin with, for any $x,y\in \calX$ define the rank-one operator $C_{xy} : \bH\to\bH$ by
$$
C_{xy}\psi := \<\psi,\Phi(y)\>\Phi(x) = \psi(y)\Phi(x).
$$
For $t\ge 0$ and $\psi\in\bH$ we further define the cross-covariance operator $C_\bH^t : \bH\to\bH$ by
$$
C_\bH^t\psi := \int\int C_{xy}\psi\,\rho_t(x,dy)\,d\mu(x) = \int (K^t\psi)(x)\Phi(x)\,d\mu(x) = \calE K^t\psi = \calE K^t\calE^*\psi.
$$
Thus, we have
\begin{align}\label{e:CKt}
C_\bH^t = \calE K^t\calE^*.
\end{align}
In other words, the cross-covariance operator $C_\bH^t$ represents the action of the Koopman semigroup through the lens of the RKHS integral operator $\mathcal{E}$ (see~\cite{klus20} for details). Being the product of the two Hilbert-Schmidt operators $\calE K^t$ and $\calE^*$, the operator $C_\bH^t$ is trace class for all $t\ge 0$ (cf.\ \cite[p.\ 521]{k}).

Note that due to $\rho_0(x,\,\cdot\,) = \delta_x$, for $t=0$ this reduces to the already introduced covariance operator
$$
\int\int C_{xy}\,\rho_0(x,dy)\,d\mu(x) = \int C_{xx}\,d\mu(x) = \calE\calE^* = C_\bH.
$$
The identity \eqref{e:CKt} shows that for all $\eta,\psi\in\bH$ we have
\begin{align}\label{e:umrechnen}
\<\eta,C_\bH^t\psi\> = \<\eta,K^t\psi\>_\mu,
\end{align}
which shows that the role of $C^t_\bH$ is analogous to that of the stiffness matrix in a traditional finite-dimensional approximation of the Koopman operator. In this analogy, the covariance operator $C_\bH$ plays the role of the mass matrix.

\subsection{Empirical estimators}\label{subsec:emp}
Recall that the resolvent set $\rho(T)$ of a bounded operator $T$, mapping from a Hilbert space $\calH$ into itself, is the set consisting of all $\la\in\C$ such that $T-\la I$ is boundedly invertible. It is the complement of the spectrum $\sigma(T)$ of $T$.

Next, we introduce empirical estimators for $C^t_\bH$ based on finite data $(x_k, y_k)$, $k=1,\ldots,m$. We consider two sampling scenarios for fixed $t>0$.

\medskip
\noindent{\bf Assumptions on the sampling scheme and the Koopman operator:}
\begin{enumerate}
\item[{\bf (1)}] The $x_k$ are drawn i.i.d.\ from $\mu$, and each $y_k\sim\mu$ is obtained from the conditional distribution $\rho_t(x_k,\cdot)$, i.e., $y_k|(x_k=x)\sim\rho_t(x,\cdot)$ for $\mu$-a.e.\ $x\in \calX$. For example, $y_k$ can be obtained by simulating the SDE~\eqref{e:SDE} starting from $x_k$ until time $t$.
\item[{\bf (2)}] $\mu$ is ergodic and both $x_k$ and $y_k$ are obtained from a single (usually long-term) simulation of the dynamics $X_t$ at discrete integration time step $t > 0$, using a sliding-window estimator, i.e.,
$$
x_0 = X_0\sim\mu,\quad x_k = X_{kt}, \quad\text{and}\quad y_k = X_{(k+1)t}.
$$
In this case, we assume that
\begin{align}\label{e:resolvent}
1\in\rho(K_0^t),    
\end{align}
where $K_0^t$ is the restriction of the Koopman operator $K^t$ to the orthogonal complement $L^2_{\mu,0}(\R^d) := \one^\perp$ of $\linspan\{\one\}$ in $L^2_\mu(\R^d)$.
\end{enumerate}

\begin{rem}
(a) The condition \eqref{e:resolvent} means that $\la=1$ is an isolated simple eigenvalue of $K^t$. It is satisfied if $K^t$ is compact. Then $\sigma(K^t)$ consists of zero and a sequence of eigenvalues converging to zero, and ergodicity ensures that the eigenvalue $\la = 1$ is simple, cf.\ Proposition \ref{p:L_and_ergo}. Another case where \eqref{e:resolvent} holds is when the semigroup $(K_0^t)_{t\ge 0}$ is exponentially stable, i.e., there exist $M\ge 1$ and $\omega>0$ such that $\|K_0^t\|\le Me^{-\omega t}$ for all $t\ge 0$. Then $\|K_0^{nt}\|^{1/n}\le M^{1/n}e^{-\omega t}$, so that the spectral radius $r = \lim_{n\to\infty}\|K_0^{nt}\|^{1/n}$ of $K_0^t$ is at most $e^{-\omega t} < 1$.

\smallskip\noindent
(b) We would like to point out that the condition \eqref{e:resolvent} does not exclude deterministic systems, i.e., autonomous ODEs $\dot x = b(x)$, in which case the operator $K^t$ is unitary on $L^2_\mu(\R^d)$.
\end{rem}

Recall that the joint distribution of two random variables $X$ and $Y$ is given by
$$
dP_{X,Y}(x,y) = dP_{Y|X=x}(y)\cdot dP_X(x).
$$
Set $X=x_k$ and $Y = y_k$. Then, in both cases {\bf (1)} and {\bf (2)}, we have $P_X = \mu$ and
$$
P_{Y|X=x}(B) = P(y_k\in B|x_k=x) = P(X_t\in B|X_0=x) = \rho_t(x,B).
$$
In other words, for the joint distribution $\mu_{0,t}$ of $x_k$ and $y_k$ we have
$$
d\mu_{0,t}(x,y) = d\rho_t(x,\cdot)(y)\cdot d\mu(x) = \rho_t(x,dy)\cdot d\mu(x).
$$
More explicitly,
$$
\mu_{0,t}(A\times B) = \int_A\rho_t(x,B)\,d\mu(x).
$$
Now, since
$$
C_\bH^t = \int\int C_{xy}\,\rho_t(x,dy)\,d\mu(x) = \int C_{xy}\,d\mu_{0,t}(x,y) = \bE\big[C_{x_k,y_k}\big],
$$
for the empirical estimator for $C_\bH^t$ we choose the expression
\begin{align}\label{e:emp_est}
\wh C_\bH^{m,t} = \frac 1m\sum_{k=0}^{m-1}C_{x_k,y_k}.
\end{align}

\section{Variance of the Empirical Estimator}\label{sec:res1}
In case {\bf (1)}, the law of large numbers \cite[Theorem 2.4]{bosq} and, in case {\bf (2)}, ergodicity~\cite{Beck57} ensures the expected behavior
$$
\lim_{m\to\infty}\|\wh C_\bH^{m,t} - C_\bH^t\|_{HS} = 0 \qquad\text{a.s.}
$$
However, this is a purely qualitative result, and nothing is known a priori on the rate of this convergence. The main result of this section, Theorem \ref{t:Eest}, yields probabilistic estimates for the expression $\|\wh C_\bH^{m,t} - C_\bH^t\|_{HS}$, see Proposition \ref{p:prob_est}. Here, our focus is on the estimation from a single ergodic trajectory, i.e., case {\bf (2)} above. While the broader line of reasoning partially resembles that of our previous paper~\cite{nueske23}, we require additional steps due to the infinite-dimensional setting introduced by the RKHS.

\begin{thm}\label{t:Eest}
The Hilbert-Schmidt variance of the empirical estimator can be written as
\begin{align}\label{e:foo}
\bE\big[\|\wh C_\bH^{m,t} - C_\bH^t\|_{HS}^2\big] = \frac 1m\left[\bE_0(t) + 2\sum_{k=1}^{m-1}\tfrac{m-k}{m}\cdot\bE\big[\<C_{x_k,y_k} - C_\bH^t,C_{x_0,y_0}-C_\bH^t\>_{HS}\big]\right],
\end{align}
where
$$
\bE_0(t) := \<K^t\vphi,\vphi\>_\mu - \|C_\bH^t\|_{HS}^2.
$$
In case {\bf (1)}, we have
\[
\bE\big[\|\wh C_\bH^{m,t} - C_\bH^t\|_{HS}^2\big] = \frac 1m\bE_0(t),
\]
whereas in case {\bf (2)},
\begin{align}\label{e:EC}
\bE\big[\|\wh C_\bH^{m,t} - C_\bH^t\|_{HS}^2\big]
&= \frac 1m\Bigg[\bE_0(t) + \sum_{i,j=1}^\infty\<Qg_{ji}^*,F_m(K_0^t)Qg_{ij}\>_\mu\Bigg],
\end{align}
where $Q$ denotes the orthogonal projection onto $L^2_{0,\mu}(\R^d) = \one^\perp$ in $L^2_\mu(\R^d)$,
$$
g_{ij} = f_j\cdot K^tf_i,\qquad g_{ji}^* = f_i\cdot (K^t)^*f_j\quad i,j\in\N,
$$
with the Mercer basis $(f_i)\subset\bH$ \braces{cf.\ Theorem \ref{t:mercer}}, and $F_m : \C\backslash\{1\}\to\C$ is given by
\[
F_m(z) = 2\sum_{k=1}^{m-1}\tfrac{m-k}{m}\cdot z^{k-1} = \frac {2}{1-z}\left(1 - \frac{1-z^m}{m(1-z)}\right),\qquad z\in\C\backslash\{1\}.
\]
\end{thm}
\begin{proof}
Let us prove \eqref{e:foo}. First of all, we set $z_k = (x_k,y_k)$, $k=1,\ldots,m$. Then
\begin{align*}
\bE\big[\|\wh C_\bH^{m,t} - C_\bH^t\|_{HS}^2\big]
&= \bE\Big[\Big\|\frac 1m\sum_{k=0}^{m-1}(C_{z_k}-C_\bH^t)\Big\|_{HS}^2\Big] = \bE\Big[ \frac 1{m^2}\sum_{k,\ell=0}^{m-1}\big\<C_{z_k} - C_\bH^t,C_{z_\ell} - C_\bH^t\big\>_{HS} \Big]\\
&= \bE\left[ \frac 1{m^2}\sum_{k=0}^{m-1}\|C_{z_k} - C_\bH^t\|_{HS}^2 + \frac 2{m^2}\sum_{k=0}^{m-1}\sum_{\ell=k+1}^{m-1}\big\<C_{z_k} - C_\bH^t,C_{z_\ell} - C_\bH^t\big\>_{HS} \right]\\
&= \frac 1m\bE\big[\|C_{z_0} - C_\bH^t\|_{HS}^2\big] + \frac 2{m^2}\sum_{k=1}^{m-1}(m-k)\bE\big[\<C_{z_k} - C_\bH^t,C_{z_0}-C_\bH^t\>_{HS}\big],
\end{align*}
where we exploited that $\bE[\<C_{z_k} - C_\bH^t,C_{z_\ell} - C_\bH^t\>_{HS}]$ only depends on the difference $\ell-k$.

Let us compute the first term. Since $\bE[C_{z_0}] = C_\bH^t$ and thus $\bE[\<C_{z_0},C_\bH^t\>_{HS}] = \|C_\bH^t\|_{HS}^2$,
$$
\bE\big[\|C_{z_0} - C_\bH^t\|_{HS}^2\big] = \bE\big[\|C_{z_0}\|_{HS}^2\big] - \|C_\bH^t\|_{HS}^2.
$$
For $\psi\in\bH$ we have
$$
\|C_{z_0}\psi\|^2 = \|\psi(y_0)\Phi(x_0)\|^2 = \psi(y_0)^2\vphi(x_0).
$$
As $\vphi(x) = k(x,x) = \sum_i f_i(x)^2$, we obtain
\begin{align*}
\bE\big[\|C_{z_0}\|_{HS}^2\big]
&= \bE\Big[\sum_i\|C_{z_0}f_i\|^2\Big] = \bE\Big[\sum_if_i(y_0)^2\vphi(x_0)\Big] = \bE[\vphi(x_0)\vphi(y_0)]\\
&= \int\int\vphi(x)\vphi(y)\,\rho_t(x,dy)\,d\mu(x) = \int\vphi(x)(K^t\vphi)(x)\,d\mu(x) = \<K^t\vphi,\vphi\>_\mu.
\end{align*}
Therefore,
$$
\bE\big[\|C_{z_0} - C_\bH^t\|_{HS}^2\big] = \<K^t\vphi,\vphi\>_\mu - \|C_\bH^t\|_{HS}^2 = E_0(t),
$$
and thus \eqref{e:foo} follows.

\bigskip
\noindent{\bf Case (1).} Since $z_k$ and $z_\ell$ are independent for $k\neq\ell$, we have $\bE\big[\<C_{z_k} - C_\bH^t,C_{z_\ell}- C_\bH^t\>_{HS}\big] = 0$. Hence, the statement of the theorem for case {\bf (1)} follows.

\bigskip
\noindent{\bf Case (2).} First of all, note that $g_{ij}\in L^2_\mu$ as
\begin{align}\label{e:gijL2}
\sum_{i,j}\int|g_{ij}|^2\,d\mu
&= \sum_{i,j}\int f_j ^2(K^tf_i)^2\,d\mu\,\le\,\sum_{i,j}\int f_j^2\cdot K^t[f_i^2]\,d\mu = \int\vphi\cdot K^t\vphi\,d\mu = \<K^t\vphi,\vphi\>_\mu.
\end{align}
As also $(K^t)^*$ enjoys the property $[(K^t)^*f]^2\le (K^t)^*f^2$ (see \ref{a:perron}), we similarly have
\begin{align}\label{e:gij*L2}
\sum_{i,j}\|g_{ji}^*\|_\mu^2\le\<K^t\vphi,\vphi\>_\mu.
\end{align}
For the cross terms, we compute
\begin{align*}
\bE\big[\<C_{z_k} - C_\bH^t,C_{z_0}-C_\bH^t\>_{HS}\big] + \|C_\bH^t\|_{HS}^2
&=\bE[\<C_{z_k},C_{z_0}\>_{HS}] = \bE\Big[\sum_i\<C_{z_k}f_i,C_{z_0}f_i\>\Big] \\
&= \bE\Big[\Big(\sum_if_i(y_k)f_i(y_0)\Big)k(x_k,x_0)\Big]\\
&= \bE\Big[\sum_{i,j}f_i(y_k)f_i(y_0)f_j(x_k)f_j(x_0)\Big].
\end{align*}
The last term can be expressed as
\begin{align*}
\int &\int\int\int \sum_{i,j}f_j(x)f_i(y)f_j(x')f_i(y')\,\rho_t(x',dy')\,\rho_{(k-1)t}(y,dx')\,\rho_t(x,dy)\,d\mu(x)\\
&= \sum_{i,j}\int\int\int f_j(x)f_i(y)f_j(x')(K^tf_i)(x')\,\rho_{(k-1)t}(y,dx')\,\rho_t(x,dy)\,d\mu(x)\\
&= \sum_{i,j}\int\int f_j(x)f_i(y)(K^{(k-1)t}g_{ij})(y)\,\rho_{t}(x,dy)\,\,d\mu(x) = \sum_{i,j}\int f_j(x)\big(K^t(f_i\cdot K^{(k-1)t}g_{ij})\big)(x)\,d\mu(x)\\
&\overset{(*)}{=} \sum_{i,j}\big\<f_i(K^t)^*f_j,K^{(k-1)t}g_{ij}\big\>_\mu = \sum_{i,j}\big\<g_{ji}^*,K^{(k-1)t}g_{ij}\big\>_\mu.
\end{align*}
For a justification of $(*)$ see Lemma \ref{l:geht}.

Let $P$ be the orthogonal projection in $L^2_\mu(\R^d)$ onto $\linspan\{\one\}$, i.e., $P = I-Q$. Then, by Lemma \ref{l:one},
\[
\bE\big[\<C_{z_k} - C_\bH^t,C_{z_0}-C_\bH^t\>_{HS}\big] + \|C_\bH^t\|_{HS}^2 = \sum_{i,j}\big\<Qg_{ji}^*,K_0^{(k-1)t}Qg_{ij}\big\>_\mu + \sum_{i,j}\big\<Pg_{ji}^*,K^{(k-1)t}Pg_{ij}\big\>_\mu.
\]
For $f\in L^2_\mu(\R^d)$ we have $Pf = \<f,\one\>_\mu\one$, so (since $\int g_{ji}^*\,d\mu = \int g_{ij}\,d\mu$)
\begin{align*}
\sum_{i,j}\big\<Pg_{ji}^*,K^{(k-1)t}Pg_{ij}\big\>_\mu
&= \sum_{i,j}\big\<\<g_{ji}^*,\one\>_\mu\one,\<g_{ij},\one\>_\mu K^{(k-1)t}\one\big\>_\mu = \sum_{i,j}|\<g_{ij},\one\>_\mu|^2 = \sum_{i,j}\left|\int f_j\cdot K^tf_i\,d\mu\right|^2\\
&= \sum_{i,j}\int\int f_j(x)(K^tf_i)(x)f_j(x')(K^tf_i)(x')\,d\mu(x)\,d\mu(x')\\
&= \sum_i\int\int k(x,x')(K^tf_i)(x)(K^tf_i)(x')\,d\mu(x)\,d\mu(x')\\
&= \sum_i\int\int\int\int \big\<f_i(y)\Phi(x),f_i(y')\Phi(x')\big\>\,\rho_t(x,dy)\,\rho_t(x',dy')\,d\mu(x)\,d\mu(x')\\
&= \sum_i\Big\|\int f_i(y)\Phi(x)\,d\mu_{0,t}(x,y)\Big\|^2 = \sum_i\|C_\bH^tf_i\|^2 = \|C_\bH^t\|_{HS}^2.
\end{align*}
At this point, we would like to remark for later use that in a similar way we get
\begin{align}\label{e:Pgs}
\sum_{i,j}\|Pg_{ij}\|_\mu^2 = \sum_{i,j}\|Pg_{ji}^*\|_\mu^2 = \|C_\bH^t\|_{HS}^2.
\end{align}
We have thus shown that
\[
\bE\big[\<C_{z_k} - C_\bH^t,C_{z_0}-C_\bH^t\>_{HS}\big] = \sum_{i,j}\big\<Qg_{ji}^*,K_0^{(k-1)t}Qg_{ij}\big\>_\mu,
\]
which concludes the proof.
\end{proof}

\begin{rem}\label{e:self-adjoint}
Let us compute the variance in the case, where the generator $\calL$ is self-adjoint with discrete spectrum. Then $\calL = \sum_{\ell=0}^\infty\mu_\ell\<\cdot,\psi_\ell\>_\mu\psi_\ell$ with eigenvalues $\mu_\ell\le 0$ and eigenfunctions $\psi_\ell$. We let $\mu_0 = 0$ and $\psi_0 = \one$. Then, setting $q_\ell = e^{\mu_\ell t}$, we get $K_0^t = \sum_{\ell=1}^\infty q_\ell\<\cdot,\psi_\ell\>_\mu\psi_\ell$ and thus
\[
F_m(K_0^t) = \sum_{\ell=1}^\infty F_m(q_\ell)\<\cdot,\psi_\ell\>_\mu\psi_\ell.
\]
It is now easy to see that (note that $g_{ji}^* = g_{ji}$ in this case)
\begin{align}\label{e:dseries}
\sum_{i,j}\<Qg_{ji},F_m(K_0^t)Qg_{ij}\>_\mu = \sum_{\ell=1}^\infty d_{\ell,t}\cdot F_m(q_\ell),
\end{align}
where $d_{\ell,t} = \sum_{i,j}\<g_{ij},\psi_\ell\>_\mu\<g_{ji},\psi_\ell\>_\mu$.
\end{rem}

In the following, we let
\[
\sigma_m^2 := \bE_0(t) + \sum_{i,j=1}^\infty\<Qg_{ji}^*,F_m(K_0^t)Qg_{ij}\>_\mu
\qquad\text{and}\qquad
\sigma_\infty^2 := \bE_0(t) + 2\sum_{i,j=1}^\infty\<Qg_{ji}^*,(I - K_0^t)^{-1}Qg_{ij}\>_\mu.
\]
Then
$$
\bE\big[\|\wh C_\bH^{m,t} - C_\bH^t\|_{HS}^2\big] = \frac{\sigma_m^2}{m}
$$
and $\sigma_m^2\to\sigma_\infty^2$ as $m\to\infty$. We can therefore interpret $\sigma^2_\infty$ as \emph{asymptotic variance} of the estimator $\wh{C}^{m,t}_\bH$, similar to our previous results in~\cite[Lemma~6]{nueske23}.

An upper bound on the variance can be obtained as follows.

\begin{cor}\label{c:error_bound_compact}
In case {\bf (2)}, for all $m\in\N$ we have
\begin{equation}\label{e:ECcor}
\sigma_m^2\,\le\,\bE_0(t)\left[1 + \|F_m(K_0^t)\|\right]\,\le\,8\bE_0(t) \|(I-K_0^t)^{-2}\|.
\end{equation}
\end{cor}
\begin{proof}
We set $P := I-Q$ and estimate (cf.\ \eqref{e:gijL2}, \eqref{e:gij*L2}, and \eqref{e:Pgs})
\begin{align*}
\sum_{i,j=1}^\infty\<Qg_{ji}^*,F_m(K_0^t)Qg_{ij}\>_\mu
&\le\Big(\sum_{i,j}\|Qg_{ji}^*\|_\mu^2\Big)^{1/2}\Big(\sum_{i,j}\|F_m(K_0^t)\|^2\|Qg_{ij}\|_\mu^2\Big)^{1/2}\\
&\le \|F_m(K_0^t)\|\Big(\sum_{i,j}\big[\|g_{ji}^*\|_\mu^2 - \|Pg_{ji}^*\|_\mu^2\big]\Big)^{1/2}\Big(\sum_{i,j}\big[\|g_{ij}\|_\mu^2 - \|Pg_{ij}\|_\mu^2\big]\Big)^{1/2}\\
&\le \|F_m(K_0^t)\|\big(\<K^t\vphi,\vphi\>_\mu - \|C_\bH^t\|_{HS}^2\big) = \|F_m(K_0^t)\|\cdot E_0(t).
\end{align*}
This proves the first estimate. For the second one, we observe that for $T := K_0^t$ we have
\begin{align*}
F_m(T) = 2(I-T)^{-1}\big(I - \tfrac 1m(I-T)^{-1}(I-T^m)\big) = 2(I-T)^{-2}\big((1-\tfrac 1m)I - T + \tfrac 1mT^m\big).
\end{align*}
Making use of $\|T\|\le 1$, we obtain
\[
\|F_m(T)\|\,\le\,2\|(I-T)^{-2}\|\cdot(1-\tfrac 1m + 1 + \tfrac 1m) = 4\|(I-T)^{-2}\|.
\]
Moreover, $\|I-T\|\le 2$, so that
\[
1 = \|(I-T)^2(I-T)^{-2}\|\le\|I-T\|^2\|(I-T)^{-2}\|\le 4\|(I-T)^{-2}\|,
\]
and the corollary is proved.
\end{proof}

\begin{rem}
\label{rem:simple_bound_sigm}
If the semigroup $(K_0^t)_{t\ge 0}$ is exponentially stable with $M\ge 1$ and $\omega>0$, then, setting $q = e^{-\omega t} < 1$, we have
\begin{align*}
\|F_m(K_0^t)\|
&= 2\Bigg\|\sum_{k=1}^{m-1}\tfrac{m-k}{m}\cdot K_0^{(k-1)t}\Bigg\|\le 2M\sum_{k=1}^{m-1}\tfrac{m-k}{m}\cdot q^{k-1} = MF_m(q)\\
&= \frac{2M}{1-q}\left(1 - \frac{1-q^m}{m(1-q)}\right)\,\le\,\frac {2M}{1-q}.
\end{align*}
Especially, if $M=1$, we obtain $1 + \|F_m(K_0^t)\|\le\frac{3-q}{1-q}$ and thus $\sigma_m^2\le\frac{3-q}{1-q}\cdot\bE_0(t)$.
\end{rem}

\begin{prop}\label{p:prob_est}
We have the following probabilistic bound on the estimation error:
\begin{numcases}
{\bP\big(\|C_\bH^t - \wh C_\bH^{m,t}\|_{HS} > \veps\big)
\le}
\frac{\sigma^2_m}{m\veps^2}, \qquad &\text{in case {\bf(2)},}\label{e:probC}
\\
\frac{E_0(t)}{m\veps^2}, \qquad &\text{in case {\bf(1)},}\label{e:probCiidM}
\\
2 \,e^{-\frac{m\veps^2}{8\|k\|_\infty^2}}, \qquad &\text{in case {\bf(1)} with bounded kernel.}\label{e:probCiidH}
\end{numcases}
In particular, the above also holds upon replacing the left-hand side by $\bP\big(\|\calE K^t\psi - \wh C_\bH^{m,t}\psi\| > \veps\big)$ for $\psi\in\bH$, $\|\psi\|=1$.
\end{prop}
\begin{proof}
The inequalities \eqref{e:probC} and \eqref{e:probCiidM} are an immediate consequence of Markov's inequality, applied to the random variable $\|C_\bH^t - \wh C_\bH^{m,t}\|_{HS}^2$. Towards the inequality \eqref{e:probCiidH}, we observe that (cf.\ the proof of Theorem \ref{t:Eest})
\[
\|C_\bH^t\|_{HS}^2 = \int\int\int\int k(x,x')k(y,y')\,\rho_t(x,dy)\,\rho_t(x',dy')\,d\mu(x)\,d\mu(x')\,\le\,\|k\|_\infty^2.
\]
and $\|C_{xy}\|_{HS}^2 = \vphi(x)\vphi(y)$ for $x,y\in\R^d$. Hence, \eqref{e:probCiidH} follows from $C_\bH^t - \wh C_\bH^{m,t} = \frac 1m\sum_{k=0}^{m-1}(C_\bH^t - C_{z_k})$, Hoeffding's inequality for Hilbert space-valued random variables \cite[Theorem 3.5]{pi} (see also \cite[Theorem A.5.2]{m}), and
\[
\|C_\bH^t - C_{xy}\|_{HS}\le \|C_\bH^t\|_{HS} + \|C_{xy}\|_{HS}\,\le\, 2\|k\|_\infty.
\]
The estimate
\begin{align*}
\|\calE K^t\psi - \wh C_\bH^{m,t}\psi\|
&= \|\calE K^t\calE^*\psi - \wh C_\bH^{m,t}\psi\| = \|(C_\bH^t - \wh C_\bH^{m,t})\psi\|\le\|C_\bH^t - \wh C_\bH^{m,t}\|_{HS}
\end{align*}
finally yields the last claim.
\end{proof}

\begin{rem}
Under additional assumptions (boundedness of the kernel, mixing, etc.), other concentration inequalities than Markov's, such as, e.g., \cite[Theorem 2.12]{bosq} ($\alpha$-mixing) or \cite[Th\'eor\`eme 3.1]{rhomari} ($\beta$-mixing), might lead to better estimates than \eqref{e:probC}.
\end{rem}

\section{Bound on the Koopman prediction error}
\label{sec:res2}

The kernel cross-covariance operator $C_\bH^t$ can also be used to approximate the predictive capabilities of the Koopman operator, for observables in $\bH$. Approximating the full Koopman operator involves the inverse of the co-variance operator, which becomes an unbounded operator on a dense domain of definition in the infinite-dimensional RKHS case. Moreover, its empirical estimator $\wh C_\bH^m$ is finite-rank and thus not even injective. While Fukumizu et al.\ tackle this problem in \cite{fsg} by means of a regularization procedure, we choose to use pseudo-inverses instead (cf.\ Remark \ref{r:pseudo}). We truncate the action of the Koopman operator using $N$~terms of the Mercer series expansion and derive a bound for the prediction error for fixed truncation parameter~$N$. While we use similar ideas as presented in~\cite{gia19}, we heavily rely on our new results on the cross-covariance operator, cf.~Section~\ref{sec:res1}. Afterwards, we deal with 
the case of Koopman-invariance of the RKHS~\cite{kss}. Here, 
we establish 
an estimate for the truncation error, which then yields a bound on the deviation from the full Koopman operator.
We emphasize that this error bound is extremely useful in comparison to its prior counterparts based on the assumption that the space spanned by a finite number of so-called observables (dictionary) is invariant under the Koopman operator. The latter essentially requires to employ only Koopman eigenfunctions as observables, see, e.g., \cite{KordMezi20, Hase21}.

Let $(e_j)$ be the Mercer orthonormal basis of $L^2_\mu(\calX)$ corresponding to the kernel $k$ and let $\la_j = \|\calE e_j\|_\mu$ as well as $f_j := \sqrt{\la_j}e_j$ (cf.\ Theorem \ref{t:mercer}). We arrange the Mercer eigenvalues in a non-increasing way, i.e.,
$$
\la_1\ge\la_2\ge\ldots.
$$
Let $\psi\in\bH$. Then
\begin{align}\label{e:aufspaltung}
K^t\psi = \sum_{j=1}^\infty\<K^t\psi,e_j\>_\mu e_j = \sum_{j=1}^\infty\<C_\bH^t\psi,e_j\>e_j = \sum_{j=1}^N\<C_\bH^t\psi,e_j\>e_j + \sum_{j=N+1}^\infty\<C_\bH^t\psi,e_j\>e_j.
\end{align}

\subsection{Prediction error}
In the next theorem, we estimate the probabilistic error between the first summand
$$
K_N^t\psi = \sum_{j=1}^N\<C_\bH^t\psi,e_j\>e_j,\qquad\psi\in\bH,
$$
and its empirical estimator, which is of the form $\sum_{j=1}^N\<\wh C_\bH^{m,t}\psi,\wh e_j\>\wh e_j$ with approximations $\wh e_j$ of the $e_j$.

\begin{thm}\label{t:main}
Assume that the eigenvalues $\la_j$ of $C_\bH$ are simple, i.e., $\la_{j+1} < \la_j$ for all $j$. Fix an arbitrary $N\in\N$ and let
\begin{equation}\label{e:delta_N}
\delta_N = \min_{j=1,\ldots,N}\frac{\la_j - \la_{j+1}}{2}.    
\end{equation}
Further, let $\veps\in (0,\delta_N)$ and $\delta\in (0,1)$ be arbitrary and fix some\footnote{In case {\bf (2)}, by Corollary \ref{c:error_bound_compact}, an amount of at least
$m\,\ge\,\max\Big\{N\,,\,\frac{16\bE_0(t)}{\veps^2\delta}\|(I-K_0^t)^{-2}\|\Big\}$ data points suffices.} $m\ge\max\{N,\frac{2\sigma_m^2}{\veps^2\delta}\}$. Let now $\wh\la_1\ge\ldots\ge\wh\la_m$ denote the largest $m$ eigenvalues of $\wh C_\bH^m$ in descending order and let $\wh e_1,\ldots,\wh e_m$ be corresponding eigenfunctions, respectively, such that $\|\wh e_j\|=\wh\la_j^{-1/2}$ for $j=1,\ldots,m$. If we define
\begin{equation}\label{e:K_est}
\wh K_N^{m,t}\psi = \sum_{j=1}^N\<\wh C_\bH^{m,t}\psi,\wh e_j\>\wh e_j,\qquad\psi\in\bH,    
\end{equation}
then, with probability at least $1-\delta$, we have that
\begin{equation}\label{e:est_err}
\|K_N^t - \wh K_N^{m,t}\|_{\bH\to L^2_\mu(\calX)}\,\le\,\left[\frac 1{\sqrt{\la_N}} + \frac{N+1}{\delta_N\la_N}(1+\|\vphi\|_1)\|\vphi\|_1^{1/2}\right]\veps.
\end{equation}
All of the above statements equally apply to case {\bf (1)} upon replacing $\sigma_m$ by $E_0(t)$.
\end{thm}

\begin{rem}\label{r:pseudo}
(a) If we set $\wh f_j = \wh\la_j^{1/2}\cdot\wh e_j$, then
$$
\wh C_\bH^m = \sum_{j=1}^m\wh\la_j\<\,\cdot\,,\wh f_j\>\wh f_j,
$$
and thus
$$
\sum_{j=1}^N\<\,\cdot\,,\wh e_j\>\wh e_j = \sum_{j=1}^N\frac 1{\wh\la_j}\<\,\cdot\,,\wh f_j\>\wh f_j = (\wh C_\bH^m)^\dagger\wh Q_N,
$$
where $\wh Q_N = \sum_{j=1}^N\<\,\cdot\,,\wh f_j\>\wh f_j$ is the orthogonal projector onto the span of the first $N$ eigenfunctions of $\wh C^{m}_\bH$ in $\bH$. Therefore,
$$
\wh K_N^{m,t}\psi = \sum_{j=1}^m\<\wh C_\bH^{m,t}\psi,\wh e_j\>\wh e_j = (\wh C_\bH^m)^\dagger\wh Q_N\wh C_\bH^{m,t}\psi.
$$
In particular, for $N=m$ we have $\wh K_N^{m,t} = (\wh C_\bH^m)^\dagger\wh C_\bH^{m,t}$, which surely is one of the first canonical choices for an empirical estimator of $K^t$.

\smallskip
\noindent(b) The functions $\wh e_j$ have unit length in the empirical $L^2_\mu$-norm:
\begin{align*}
\frac{1}{m}\sum_{k=1}^m \wh e_j(x_k) \wh e_j(x_k) &= \innerprod{\wh C^{m}_\bH \wh e_j}{\wh e_j} = 1.
\end{align*}
Therefore, projecting onto the first $N$ empirical Mercer features is the \emph{whitening transformation} commonly used in traditional EDMD~\cite{klus18_rev}.

\end{rem}

\begin{proof}[Proof of Theorem \rmref{t:main}]
By Proposition \ref{p:prob_est}, both events $\|C_\bH^t - \wh C_\bH^{m,t}\|_{HS}\le\veps$ and $\|C_\bH - \wh C_\bH^{m}\|_{HS}\le\veps$ occur with probability at least $1-\delta/2$, respectively. Hence, they occur simultaneously with probability at least $1-\delta$.

In the remainder of this proof we assume that both events occur. Then all the statements deduced in the following hold with probability at least $1-\delta$.

Let us define the intermediate approximation
$$
\wt K_N^{m,t}\psi = \sum_{j=1}^N\<\wh C_\bH^{m,t}\psi,e_j\>e_j,\qquad\psi\in\bH.
$$
Let $\psi\in\bH$ be arbitrary. Setting $C := C_\bH^t - \wh C_\bH^{m,t}$, we have
\begin{align*}
\|K_N^t\psi - \wt K_N^{m,t}\psi\|_\mu^2
&= \Bigg\| \sum_{j=1}^N\big\<C\psi,e_j\big\>e_j \Bigg\|_\mu^2 = \sum_{j=1}^N\big|\big\<C\psi,e_j\big\>\big|^2 = \sum_{j=1}^N\big|\big\<\psi,C^*e_j\big\>\big|^2\\
&\le \|\psi\|^2\sum_{j=1}^N\|C^*e_j\|^2 \le \|\psi\|^2\sum_{j=1}^N\frac 1{\la_j}\|C^*f_j\|^2\le \frac{\|\psi\|^2}{\la_N}\sum_{j=1}^N\|C^*f_j\|^2\\
&\le \frac{\|\psi\|^2}{\la_N}\sum_{j=1}^\infty\|C^*f_j\|^2 = \frac{\|\psi\|^2}{\la_N}\cdot\|C_\bH^t - \wh C_\bH^{m,t}\|_{HS}^2,
\end{align*}
and thus,
$$
\|K_N^t\psi - \wt K_N^{m,t}\psi\|_\mu\,\le\,\frac{\|\psi\|}{\sqrt{\la_N}}\cdot\veps.
$$
Next, we aim at estimating the remaining error
\begin{align*}
\wt K_N^{m,t}\psi - \wh K_N^{m,t}\psi
&= \sum_{j=1}^N\<\wh C_\bH^{m,t}\psi,e_j\>e_j - \sum_{j=1}^N\<\wh C_\bH^{m,t}\psi,\wh e_j\>\wh e_j\\
&= \sum_{j=1}^N\la_j^{-1}\<\wh C_\bH^{m,t}\psi,f_j\>f_j - \sum_{j=1}^N\wh\la_j^{-1}\<\wh C_\bH^{m,t}\psi,\wh f_j\>\wh f_j\\
&= \sum_{j=1}^N\la_j^{-1}\<f,f_j\>f_j - \sum_{j=1}^N\wh\la_j^{-1}\<f,\wh f_j\>\wh f_j\\
&= \sum_{j=1}^N\big[\la_j^{-1}P_jf - \wh\la_j^{-1}\wh P_jf\big]\\
&= \sum_{j=1}^N\la_j^{-1}(P_j - \wh P_j)f + \sum_{j=1}^N(\la_j^{-1}-\wh\la_j^{-1})\wh P_jf,
\end{align*}
where $f = \wh C_\bH^{m,t}\psi$,
$$
P_jf = \<f,f_j\>f_j
\qquad\text{and}\qquad
\wh P_jf = \<f,\wh f_j\>\wh f_j.
$$
By \eqref{e:bounded_embedding}, it suffices to estimate the above error in the $\|\cdot\|$-norm. By Theorem \ref{t:dk_special}, the first summand can be estimated as
\begin{align*}
\Big\|\sum_{j=1}^N\la_j^{-1}(P_j - \wh P_j)f\Big\|
&\le \sum_{j=1}^N\frac{1}{\la_j}\|P_j - \wh P_j\|\|f\|\,\le\,\frac{N\cdot\|C_\bH - \wh C_\bH^m\|}{\la_N\delta_N}\,\|f\|\,\le\,\frac{N}{\la_N\delta_N}\,\|f\|\veps.
\end{align*}
For the second summand we have
$$
\Big\|\sum_{j=1}^N(\la_j^{-1}-\wh\la_j^{-1})\wh P_jf\Big\|^2 = \sum_{j=1}^N|\la_j^{-1}-\wh\la_j^{-1}|^2\|\wh P_jf\|^2 = \sum_{j=1}^N\frac{|\la_j-\wh\la_j|^2}{\la_j^2\wh\la_j^2}\|\wh P_jf\|^2.
$$
Now, note that  $\epsilon < \delta_N$ by assumption and therefore $\|C_\bH - \wh C_\bH^m\|_{HS}\le\delta_N\le\frac{\la_N-\la_{N+1}}{2}\le\frac{\la_N}2$. For $j=1,\ldots,N$, according to Theorem \ref{t:eigs} this implies
$$
\wh\la_j\,\ge\,\la_j - |\la_j - \wh\la_j|\,\ge\,\la_j - \|C_\bH - \wh C_\bH^m\|_{HS}\,\ge\,\la_j - \tfrac{\la_N}2\,\ge\,\tfrac{\la_j}2.
$$
Hence,
$$
\Big\|\sum_{j=1}^N(\la_j^{-1}-\wh\la_j^{-1})\wh P_jf\Big\|^2\,\le\,4\sum_{j=1}^N\frac{|\la_j-\wh\la_j|^2}{\la_j^4}\|\wh P_jf\|^2\,\le\,4\frac{\|C_\bH - \wh C_\bH^m\|_{HS}^2}{\la_N^4}\|\wh Q_N f\|^2,
$$
and thus,
$$
\Big\|\sum_{j=1}^N(\la_j^{-1}-\wh\la_j^{-1})\wh P_jf\Big\|\,\le\,\frac{2}{\la_N^2}\|f\|\veps\,\le\,\frac{1}{\la_N\delta_N}\|f\|\veps.
$$
From
$$
\|\wh C_\bH^{m,t}\|\le\|\wh C_\bH^{m,t} - C_\bH^t\| + \|C_\bH^t\|\le \|\wh C_\bH^{m,t} - C_\bH^t\|_{HS} + \|\calE K^t\calE^*\|\le\veps+\|\vphi\|_1
$$
we conclude
\begin{align*}
\big\|\wt K_N^{m,t}\psi - \wh K_N^{m,t}\psi\big\|
&\le \frac{N+1}{\la_N\delta_N}\|\wh C_\bH^{m,t}\psi\|\veps\,\le\,\frac{N+1}{\la_N\delta_N}(\veps + \|\vphi\|_1)\|\psi\|\veps.
\end{align*}
All together, we obtain (recall \eqref{e:bounded_embedding})
\begin{align*}
\|K_N^t\psi - \wh K_N^{m,t}\psi\|_\mu
&\le \|K_N^t\psi - \wt K_N^{m,t}\psi\|_\mu + \|\vphi\|_1^{1/2}\|\wt K_N^{m,t}\psi - \wh K_N^{m,t}\psi\|\\
&\le \frac{\|\psi\|}{\sqrt{\la_N}}\cdot\veps + \frac{N+1}{\la_N\delta_N}(\veps + \|\vphi\|_1)\|\vphi\|_1^{1/2}\|\psi\|\veps\\
&= \left[\frac 1{\sqrt{\la_N}} + \frac{N+1}{\delta_N\la_N}(1+\|\vphi\|_1)\|\vphi\|_1^{1/2}\right]\veps\cdot\|\psi\|,
\end{align*}
which implies \eqref{e:est_err}.
\end{proof}

\subsection{Projection error in case of Koopman-invariance of the RKHS}\label{ss:invariant}
In the preceding section, we have seen that the empirical operator $\wh K^{m, t}_N$ can be written as $(\wh C^{m}_\bH)^\dagger \wh C^{m, t}_\bH$ if $m = N$. In the limit $m \rightarrow \infty$, we would arrive at the operator $C_\bH^{-1} C^t_\bH$, which is not even well-defined for all $\psi\in\bH$, in general. However, if the RKHS is invariant under $K^t$, the above operator limit is well-defined as a bounded operator on $\bH$. In this situation we are able to extend Theorem \ref{t:main} to an estimate on the full error made by our empirical estimator.

We start by defining the operator
$$
K_\bH^t := C_\bH^{-1}C_\bH^t
$$
on its natural domain
\begin{equation}\label{e:domain}
\dom K_\bH^t := \{\psi\in\bH : C_\bH^t\psi\in\ran C_\bH\}.
\end{equation}
We consider $K_\bH^t$ as an operator from $\bH$ into itself (with domain of definition in $\bH$).

\begin{lem}\label{l:closed}
We have
\begin{equation}\label{e:domain2}
\dom K_\bH^t = \{\psi\in\bH : K^t\psi\in\bH\},
\end{equation}
and $K_\bH^t$ is closed.
\end{lem}
\begin{proof}
Note that $C_\bH^t\psi\in\ran C_\bH$ if and only if $\calE K^t\psi = C_\bH\phi$ for some $\phi\in\bH$. Since $\C_\bH\phi = \calE\phi$ and $\ker\calE = \{0\}$, the latter is equivalent to $K^t\psi = \phi\in\bH$, which proves the representation of the domain. As to the closedness of $K_\bH^t$, let $(\psi_n)\subset\dom K_\bH^t$ and $\phi\in\bH$ such that $\psi_n\to\psi$ in $\bH$ and $K_\bH^t\psi_n\to\phi$ in $\bH$ as $n\to\infty$. The latter implies $C_\bH^t\psi_n\to C_\bH\phi$, while the first implies $C_\bH^t\psi_n\to C_\bH^t\psi$ in $\bH$ as $n\to\infty$, from which we conclude that $C_\bH^t\psi = C_\bH\phi$, i.e., $\psi\in\dom K_\bH^t$ and $K_\bH^t\psi = \phi$.
\end{proof}

If the Koopman operator leaves the RKHS $\bH$ invariant (i.e., $K^t\bH\subset\bH$), $K^t_\bH$ is defined on all of $\bH$. Moreover, since the canonical inclusion map $\calE^* : \bH\to L_\mu^2(\calX)$ is injective, it possesses an unbounded inverse on its range $\bH$, and therefore:
\begin{align}
\label{eq:invariant_kernel_koop}
C_\bH^{-1} C^t_\bH \phi &= C_\bH^{-1}\calE K^t \calE^* \phi = (\calE \calE^*)^{-1} \calE \calE^* (\calE^*)^{-1} K^t \calE^* \phi = (\calE^*)^{-1}K^t\calE^*\phi.
\end{align}
Remarkably, invariance of $\bH$ under the Koopman operator implies that the left-hand side not only reproduces the Koopman operator on $\bH$, but actually defines a bounded operation.

Parts of the next proposition can be found in \cite[Theorem 5.3]{kss} and~\cite[Theorem 1]{Douglas1966}.

\begin{prop}
For $t>0$, the following statements are equivalent:
\begin{enumerate}
\item[{\rm (i)}]   $K^t\bH\subset\bH$.
\item[{\rm (ii)}]  $K_\bH^t\in L(\bH)$.
\item[{\rm (iii)}] $\ran C_\bH^t\subset\ran C_\bH$.
\end{enumerate}
\end{prop}
\begin{proof}
With regard to the two representations \eqref{e:domain} and \eqref{e:domain2} of the domain, it is immediate that both (i) and (iii) are equivalent to $\dom K_\bH^t = \bH$. The equivalence of the latter to (ii) follows from the closed graph theorem.
\end{proof}

Note that if one of (i)--(iii) holds, then $K_\bH^t = K^t|_\bH$.

\begin{thm}\label{t:proj_err}
In addition to the assumptions in Theorem \ref{t:main}, assume that $\bH$ is invariant under the Koopman operator $K^t$. For fixed $N\in\N$, let $\delta_N$ be as in \eqref{e:delta_N}, choose $\veps$, $\delta$, and $m$ as in Theorem \ref{t:main} and define the empirical estimator $\wh K_N^{m,t}$ as in \eqref{e:K_est}. Then, with probability at least $1-\delta$ we have that
\begin{equation}\label{e:full_est}
\|K^t - \wh K_N^{m,t}\|_{\bH\to L^2_\mu(\calX)}\,\le\,\sqrt{\la_{N+1}}\,\|K_\bH^t\| + \left[\frac 1{\sqrt{\la_N}} + \frac{N+1}{\delta_N\la_N}(1+\|\vphi\|_1)\|\vphi\|_1^{1/2}\right]\veps.
\end{equation}
\end{thm}
\begin{proof}
First of all, Theorem \ref{t:main} implies that
\begin{align*}
\|K^t - \wh K_N^{m,t}\|_{\bH\to L^2_\mu(\calX)}
&\le\|K^t - K_N^t\|_{\bH\to L^2_\mu(\calX)} + \|K_N^t - \wh K_N^{m,t}\|_{\bH\to L^2_\mu(\calX)}\\
&\le \|K^t - K_N^t\|_{\bH\to L^2_\mu(\calX)} + \left[\frac 1{\sqrt{\la_N}} + \frac{N+1}{\delta_N\la_N}(1+\|\vphi\|_1)\|\vphi\|_1^{1/2}\right]\veps.
\end{align*}
Now, for $\psi\in\bH$,
\begin{align*}
\|K^t\psi - K_N^t\psi\|_\mu^2
&= \Bigg\|\sum_{j=N+1}^\infty\<C_\bH^t\psi,e_j\>e_j\Bigg\|_\mu^2 = \sum_{j=N+1}^\infty|\<C_\bH^t\psi,e_j\>|^2 = \sum_{j=N+1}^\infty\frac{1}{\la_j}|\<C_\bH^t\psi,f_j\>|^2\\
&= \sum_{j=N+1}^\infty\frac{1}{\la_j}|\<K_\bH^t\psi,C_\bH f_j\>|^2 = \sum_{j=N+1}^\infty\la_j|\<K_\bH^t\psi,f_j\>|^2 \le \la_{N+1}\|K_\bH^t\psi\|^2,
\end{align*}
which proves the theorem.
\end{proof}

\begin{rem}
(a) The proof of Theorem \ref{t:proj_err} shows that the projection error $\|K^t\psi - K_N^t\psi\|_\mu$ decays at least as fast as the square roots of the eigenvalues of $C_\bH$. Recall that $(\la_j)_{j\in\N}\in\ell^1(\N)$, since $C_\bH$ is trace class with $\sum_{j=1}^\infty\la_j = \Tr(C_\bH) = \|\calE^*\|_{HS}^2 = \|\vphi\|_1$, see Lemma \ref{l:ECH}(c).

\smallskip
\noindent
(b) In \ref{s:OU_invariance}, we prove that the RKHS generated by Gaussian RBF kernels on $\R$ is invariant under the Koopman semigroup associated with the 1D Ornstein-Uhlenbeck process. In fact, it can be proved that this invariance also holds in higher dimensions. This shows that the assumption in Theorem \ref{t:proj_err} is not too exotic and can be satisifed.
\end{rem}

\section{Illustration with the Ornstein-Uhlenbeck process}
\label{sec:numerics}
For the numerical illustration of our results, we consider the Ornstein-Uhlenbeck (OU) process on $\calX = \R$, which is given by the SDE
$$
dX_t = -\alpha X_t\,dt + dW_t,
$$
where $\alpha > 0$ is a positive parameter.

\subsection{Analytical Results}
Since all relevant properties of the OU process are available in analytical form, we can exactly calculate all of the terms appearing in our theoretical error bounds. Moreover, we can also compute the exact estimation and prediction errors for finite data in closed form. Let us begin by recapping the analytical results required for our analysis, which can be found in~\cite{pav14}.

The invariant measure $\mu$, and the density of the stochastic transition kernel $\rho_t$, are given by
\begin{align*}
d\mu(x) = \sqrt{\frac{\alpha}{\pi}} e^{-\alpha x^2}\,dx
\qquad\text{and}\qquad
\rho_t(x,dy) = \sqrt{\frac{\alpha}{\pi v_t^2}}\exp\Big[-\frac{\alpha}{v_t^2}(y - e^{-\alpha t}x)^2\Big]\,dy,
\end{align*}
with $v_t^2 = 1 - e^{-2\alpha t}$. The Koopman operators $K^t$ are self-adjoint in $L^2_\mu(\R)$, their eigenvalues and corresponding eigenfunctions are given by
\[
q_j = e^{-\alpha j t}
\qquad\text{and}\qquad
\psi_j(x) = \frac{1}{\sqrt{2^j \alpha^j j!}}H_j(\sqrt{2\alpha}x),\quad j\in\N_0,
\]
where $H_j$ are the physicist's Hermite polynomials.

We consider the Gaussian radial basis function (RBF) kernel with bandwidth $\sigma>0$, i.e.,
\begin{equation*}
k(x, y) = \exp\left[-\frac{(x-y)^2}{\sigma^2}\right].
\end{equation*}
Let us quickly verify that this choice of the kernel satisfies the compatibility assumptions (A1)--(A3). Indeed, (A1) is trivial as $k(x,x) = 1$ and (A3) follows easily from the continuity of the functions in $\bH$. To see that $\bH$ is dense in $L^2_\mu(\R)$ (i.e., (A2)), let $\psi\in L^2_\mu(\R)$ be such that $\<\psi,\Phi(y)\>_\mu=0$ for all $y\in\R$. The latter means that $\phi * \vphi_\sigma = 0$, where $\phi(x) = \psi(x)e^{-\alpha x^2}$ and $\vphi_\sigma(x) = e^{-x^2/\sigma^2}$. We apply the Fourier transform and obtain $\wh\phi\cdot\wh{\vphi_\sigma} = 0$. Noting that the Fourier transform of a Gaussian is again a Gaussian, we get $\wh\phi=0$ and thus $\psi=0$.

The Mercer eigenvalues and features with respect to the invariant measure $\mu$ of the Ornstein-Uhlenbeck process, i.e., the eigenvalues and eigenfunctions of the integral operator $\mathcal{E}^* \mathcal{E}$ in $L^2_\mu(\R)$, are also available in analytical form~\cite{fass12}. They are given by
\begin{align*}
\lambda_i = \sqrt{\frac{\alpha}{C_1}} \left[\frac{1}{\sigma^2 C_1}\right]^{i}
\qquad\text{and}\qquad
\varphi_i(x) = \gamma_i e^{-\zeta^2 x^2}H_{i}\left(\sqrt{\alpha}\eta x\right),\quad i\in\N_0,
\end{align*}
using the following constants:
\begin{align*}
\eta &= \left[1 + \frac{4}{\alpha\sigma^2}\right]^{1/4}, & \gamma_i &= \left[\frac\eta{2^i\Gamma(i+1)}\right]^{1/2}, & \zeta^2 &= \frac{\alpha}{2}(\eta^2 - 1), & C_1 &= \alpha + \zeta^2 + \sigma^{-2}.
\end{align*}
With these results, we can compute the variance of the empirical estimator for $C^t_\bH$ as described in Theorem~\ref{t:Eest}. The eigenvalues $q_j$ were already given above. The coefficients $d_{j,t}$ (cf.\ Remark \ref{e:self-adjoint}) are given by
\begin{align*}
d_{j,t}
&= \sum_{k,\ell} \lambda_k\la_\ell \left[\int \vphi_k(x)\vphi_\ell(y)\psi_j(x) \,d\mu_{0,t}(x,y)\right] \left[\int \vphi_\ell(x)\vphi_k(y)\psi_j(x) \,d\mu_{0,t}(x,y)\right].
\end{align*}
The series needs to be truncated at a finite number of terms and the integrals can be calculated by numerical integration. As furthermore (see the proof of Theorem \ref{t:Eest})
\begin{align}\label{e:ct_hsnorm_ou}
\|C^t_\bH \|^2_{HS}
&= \sum_{k,\ell}|\<g_{k\ell},\one\>_\mu|^2 = \sum_{k,\ell} \lambda_k \lambda_\ell \left[\int \varphi_k(x)\varphi_\ell(y) \,d\mu_{0,t}(x, y)\right]^2
\end{align}

the Hilbert-Schmidt norm of the cross-covariance operator $C^t_\bH$ can be computed similarly. Since, for the Gaussian RBF kernel, we have $\varphi(x) = k(x, x) = 1$ for all $x$, we therefore find
\begin{align*}
\mathbb{E}_0(t) = \innerprod{K^t \varphi}{\varphi}_\mu - \|C^t_\bH\|^2_{HS} = 1 -  \|C^t_\bH\|^2_{HS},
\end{align*}
completing the list of terms required by Theorem~\ref{t:Eest} and Remark~\ref{rem:simple_bound_sigm}. In addition, we notice that upon replacing either one or two of the integrals in~\eqref{e:ct_hsnorm_ou} by finite-data averages, we can also calculate $\|\wh{C}^{m, t}_\bH \|^2_{HS}$ and $\<C^t_\bH,\wh{C}^{m, t}_\bH\>_{HS}$. Therefore, the estimation error for finite data $\{(x_k, y_k)\}_{k=1}^m$ can be obtained by simply expanding the inner product
\[
\|C^t_\bH - \wh{C}^{m, t}_\bH \|^2_{HS} = \|C^t_\bH \|^2_{HS} + \|\wh{C}^{m, t}_\bH \|^2_{HS} - 2\<\wh{C}^{m, t}_\bH,C^t_\bH\>_{HS},
\]
allowing us to precisely compare the estimation error to the error bounds obtained in Theorem~\ref{t:Eest}. 

Besides the estimation error for $C^t_\bH$, we are also interested in the prediction error, which is bounded according to Theorem~\ref{t:main}. We will compare these bounds to the actual error $\|(K^t_N - \wh{K}^{m,t}_N) \phi \|_{L^2_\mu(\calX)}$, for a specific observable $\phi \in \bH$ and a fixed number of  $N$ Mercer features. For the OU process, it is again beneficial to consider Gaussian observables $\phi$:
\[ \phi(x) = \frac{1}{\sqrt{2\pi \sigma_0^2}} \exp\left[-\frac{(x - m_0)^2}{2\sigma_0^2}\right]. \]
Application of the Koopman operator leads to yet another, unnormalized Gaussian observable, which is given by
\begin{align*}
K^t \phi(x) &= \frac{1}{\sqrt{2\pi \sigma_t^2}} \exp\left[-\frac{(m_0 - e^{-\alpha t}x)^2}{2\sigma_t^2} \right], & \sigma_t^2 = \sigma_0^2 + v_t^2.
\end{align*}
The inner products of $K^t \phi$ with the Mercer eigenfunctions $\varphi_i$ can be evaluated by numerical integration, providing full access to the truncated observable $K^t_N \phi$. On the other hand, the empirical approximation $\wh{K}^{m, t}_N \phi$ can be computed directly based on the data. We note that
\begin{align*}
\wh{K}^{m, t}_N \phi &= \sum_{j=1}^N \innerprod{\wh{C}^{m, t}_\bH \phi}{\wh{e}_j}\wh{e}_j
= \frac{1}{m}\sum_{k=1}^m \phi(y_k) \sum_{j=1}^N\innerprod{\Phi(x_k)}{\wh{e}_j}\wh{e}_j = \frac{1}{m}\sum_{k=1}^m \phi(y_k) \sum_{j=1}^N \wh{e}_j(x_k)\wh{e}_j.
\end{align*}
The functions $\wh{e}_j$ can be obtained from the eigenvalue decomposition of the standard kernel Gramian matrix
\[
\frac{1}{m}K_\calX := \frac{1}{m}\left[k(x_k, x_l)\right]_{k,l=1}^m,
\]
as the latter is the matrix representation of the empirical covariance operator $\wh{C}^m_\bH$ on the subspace $\mathrm{span}\{\Phi(x_k) \}_{k=1}^m$. If $\frac{1}{m}K_\calX = V\Lambda V^\top$ is the spectral decomposition of the Gramian, then
\[
\wh{e}_j = \frac{1}{m^{1/2} \wh{\lambda}_j}\sum_{l=1}^m V_{lj} \Phi(x_l)
\]
are the correctly normalized eigenfunctions according to Theorem~\ref{t:main}. Plugging this into the above, we find
\begin{align*}
\wh{K}^{m, t}_N \phi(x) &= \frac{1}{m}\sum_{k=1}^m \phi(y_k) \sum_{j=1}^N  \frac{1}{m^{1/2}\wh{\lambda}_j}\sum_{l=1}^m V_{lj} k(x_l, x_k)  \frac{1}{m^{1/2}\wh{\lambda}_j}\sum_{r=1}^m V_{rj} k(x_r, x) \\
&= \frac{1}{m} \phi(Y)^\top \frac{1}{m}K_\calX \left[V_N \Lambda_N^{-2} V_N^\top\right] K_{\calX, x} \\
&= \frac{1}{m} \phi(Y)^\top V_N \Lambda_N^{-1} V_N^\top K_{\calX, x},
\end{align*}
where $\phi(Y) = [\phi(y_k)]_{k=1}^m$, $K_{\calX,x} = [k(x_k,x)]_{k=1}^m$, $V_N = V[I_N\;0_{m-N}]^\top$, $\Lambda_N = \diag(\wh\la_j)_{j=1}^N$.

\subsection{Numerical Results}
For the actual numerical experiments, we set $\alpha = 1$, choose the Koopman lag time as $t = 0.05$, and downsample all simulation data such that successive time steps are separated by time $t$. We compute the exact variance $\mathbb{E}[\|C^t_\bH - \wh{C}^{m, t}_\bH \|^2_{HS}]$ by the expression given in Theorem~\ref{t:Eest}, and also the coarser estimate for the variance given in Remark~\ref{rem:simple_bound_sigm}. In addition, we also compute the i.i.d. variance $\frac{1}{m}\mathbb{E}_0(t)$. We test three different kernel bandwidths, $\sigma \in \{0.05, 0.1, 0.5 \}$. All Mercer series are truncated after the first 10 terms for $\sigma \in \{0.1, 0.5 \}$, and 20 terms for $\sigma = 0.05$, while Koopman eigenfunction expansions are truncated after 15 terms.

In the first set of experiments, we use Chebyshev's inequality as in Proposition~\ref{p:prob_est} combined with the variance estimates described above to compute the maximal estimation error $\|C^t_\bH - \wh{C}^{m, t}_\bH \|_{HS}$ that can be guaranteed with confidence $1 - \delta = 0.9$, for a range of data sizes $m$ between $m = 20$ and $m = 50.000$. As a comparison, we generate $200$ independent simulations with simulation horizon $m\cdot t$, for each data size $m$. We then compute the resulting estimation error using the expressions given in the previous section. The comparison of these results for all data sizes $m$ and the different kernel bandwidths is shown in Figure~\ref{fig:estimation_error_ou}. We observe that the bound based on the exact variance from Theorem~\ref{t:Eest} is quite accurate, over-estimating the actual error by about a factor three, and captures the detailed qualitative dependence of the estimation error on $m$ and $\sigma$. The coarser bound from Remark~\ref{rem:simple_bound_sigm}, however, appears to discard too much information, it over-estimates the error by at least an order of magnitude, and also does not change significantly with $\sigma$. Finally, we note that for larger kernel bandwidth, the i.i.d.\ variance is indeed too small, leading to an under-estimation of the error. This observation confirms that it is indeed necessary to take the effect of the correlation between data points into account.

\begin{figure}[htb]
\centering
\includegraphics[width=0.48\textwidth]{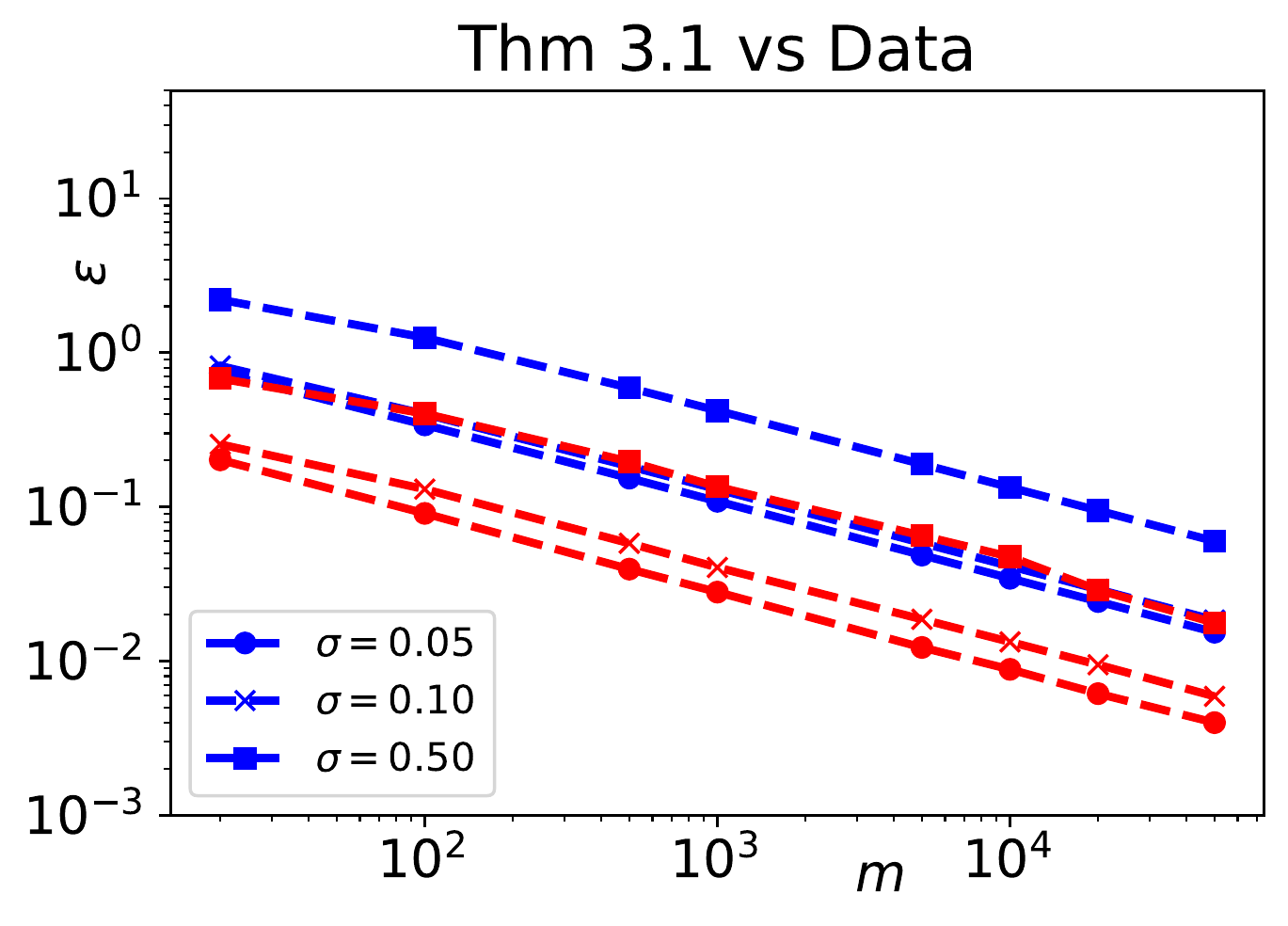}
\includegraphics[width=0.48\textwidth]{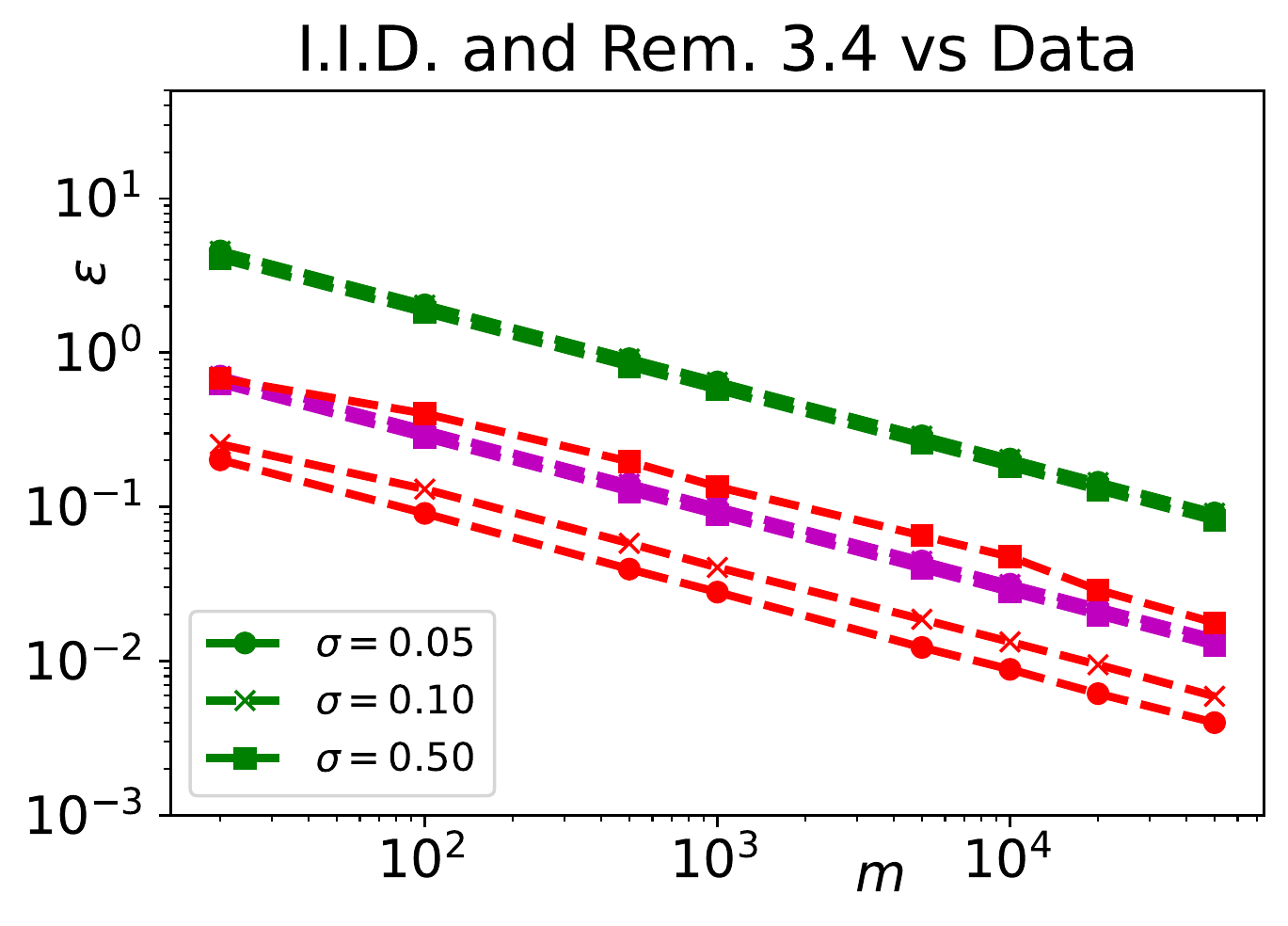}
\caption{Probabilistic error estimates for $C^t_\bH$ associated to the OU process, at lag time $t = 0.05$, and the Gaussian RBF kernel with different bandwidths $\sigma \in \{0.05, 0.1, 0.05 \}$ (indicated by circles, x-es, and squares). We show the estimated error obtained from Proposition~\ref{p:prob_est}, with confidence $1-\delta = 0.9$, using either the exact variance given in Theorem~\ref{t:Eest} (blue), the coarser estimate in Remark~\ref{rem:simple_bound_sigm} (green), or the i.i.d.-variance $\frac{1}{m} \mathbb{E}_0(t)$ (purple). The red curves in both panels shows the $0.9$-percentile of the estimation error based on $200$ independent simulations.}
\label{fig:estimation_error_ou}
\end{figure}
\begin{figure}[htb]
\centering
\includegraphics[width=0.45\textwidth]{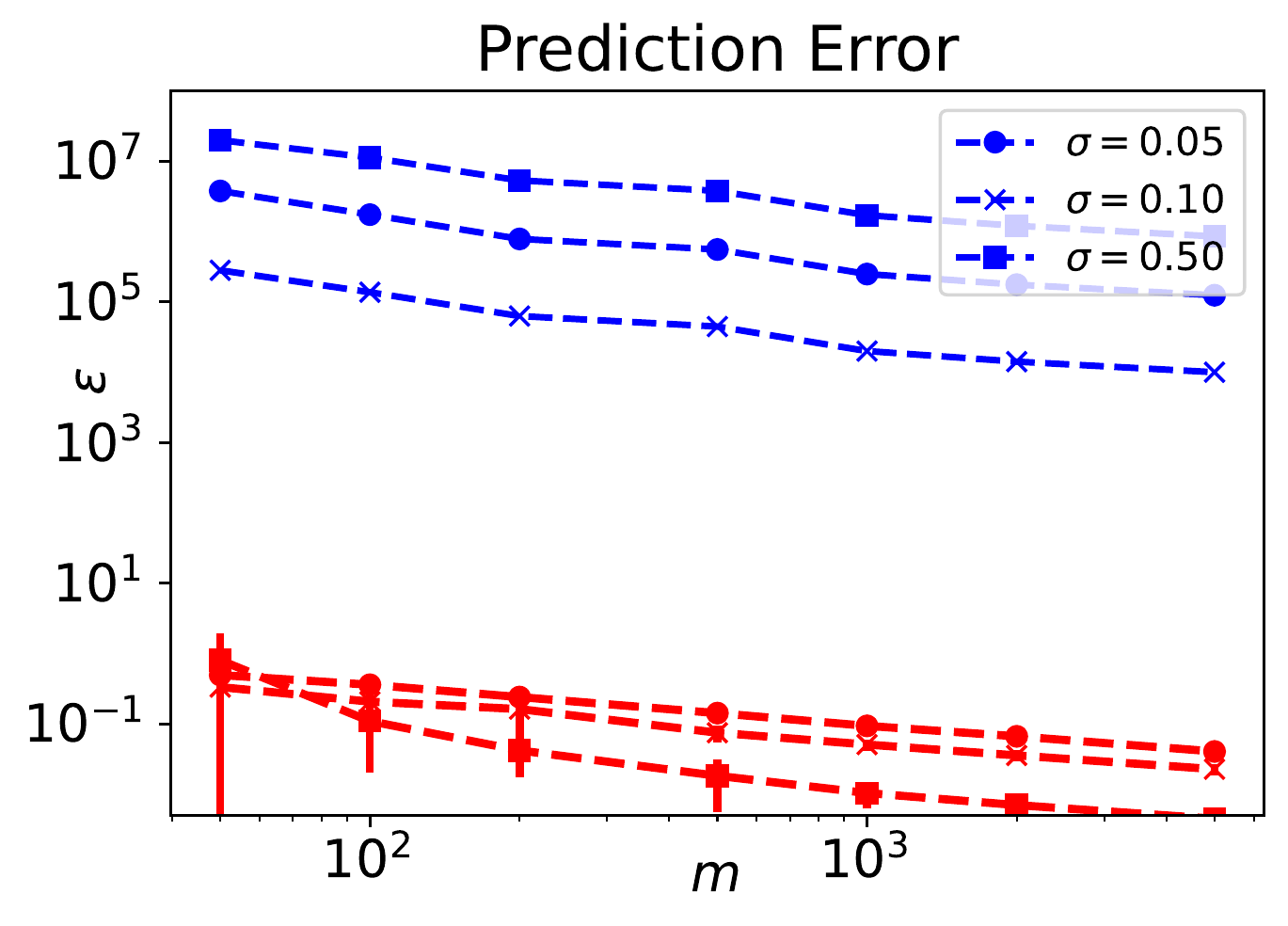}
\caption{Comparison of the theoretical bound on the prediction error $\|K^t_N\phi - \wh{K}^{m,t}_N\phi \|_\mu$, if $\phi$ is chosen as the first Mercer feature $\varphi_0$, using $N = 20$ (for $\sigma = 0.05$) or $N = 10$ (otherwise) in the Mercer series representation. The predicted error is shown in blue, different bandwidths are indicates by circles, x-es and squares. Error bars for the actual error obtained from 20 independent data sets are shown in red.}
\label{fig:prediction_error_ou}
\end{figure}

In a second set of experiments, we test the performance of our theoretical bounds concerning the prediction of expectations for individual observables, obtained in Theorem~\ref{t:main}. For the same three Gaussian RBF kernels as in the first set of experiments, we consider the observable $\phi = \varphi_0$, i.e., the first Mercer feature. As above, we choose $N = 10$ or $N = 20$, depending on the bandwidth, to truncate the Mercer series expansion $K^t_N \phi$ and its empirical approximation $\wh{K}^{m, t}_N \phi$. Note that $\phi$ is a different observable depending on the bandwidth. Again, we set $1- \delta = 0.9$, and use Theorem~\ref{t:main} to bound the $L^2_\mu$-error between $K^t_N \phi$ and $\wh{K}^{m, t}_N\phi$. As a comparison, we compute the actual $L^2_\mu$-error by numerical integration, using the fact that we can evaluate $K^t_N \phi(x)$ and $\wh{K}^{m, t}_N\phi(x)$ at any $x$ based on the discussion above. We repeat this procedure 20 times and provide average errors and standard deviations. The results  for all three kernels are shown in Figure~\ref{fig:prediction_error_ou}, and we find that our theoretical bounds are much too pessimistic in all cases. This finding highlights our previous observation that bounding the prediction error outside the RKHS still requires more in-depth research.

\section{Conclusions}
\label{sec:conclusions}
We have analyzed the finite-data estimation error for data-driven approximations of the Koopman operator on reproducing kernel Hilbert spaces. More specifically, we have provided an exact expression for the variance of empirical estimators for the cross-covariance operator, if a sliding-window estimator is applied to a long ergodic trajectory of the dynamical system (Theorem \ref{t:Eest}). This setting is relevant for many complex systems, such as molecular dynamics simulations. Our results present a significant improvement over the state of the art, since they concern a setting where the notorious problem of dictionary selection can be circumvented, and therefore no longer depend on the dictionary size. We have also extended the concept of asymptotic variance to an infinite-dimensional approximation space for the Koopman operator. Our numerical study on the Ornstein Uhlenbeck process has shown that, even using a simple mass concentration inequality, accurate bounds on the estimation error can be obtained (Figure \ref{fig:estimation_error_ou}).

In our second main contribution, Theorem \ref{t:main}, we have extended our estimates to a uniform bound on the prediction error for observables in the RKHS. Thereby, we have circumvented dealing with an unbounded inverse of the covariance operator by applying a finite-dimensional truncation of the associated Mercer series. In case of Koopman-invariance of the RKHS, we were even able to find a bound on the truncation error which then yields estimates for the full approximation error (Theorem \ref{t:proj_err}). The resulting error bounds have, however, proven very conservative in the numerical examples (Figure \ref{fig:prediction_error_ou}). Therefore, obtaining sharper bounds on the prediction error constitutes a primary goal for future research.

\medskip
\noindent{\bf Acknowledgments.} F. Philipp was funded by the Carl Zeiss Foundation within the project DeepTurb--Deep Learning in and from Turbulence. He was further supported by the free state of Thuringia and the German Federal Ministry of Education and Research (BMBF) within the project THInKI--Th\"uringer Hochschulinitiative für KI im Studium. 
K. Worthmann gratefully acknowledges funding by the German Research Foundation (DFG grant WO 2056/14-1, project no 507037103).

\section*{Data Availability}
Codes and data required for reproducing the numerical results in this study are available at \url{https://zenodo.org/doi/10.5281/zenodo.10044360}.

\bigskip
\appendix
\section{Proofs}\label{a:proofs}
\begin{proof}[Proof of Lemma \ref{l:basics0}]
Let $\psi\in\bH$. Then \eqref{e:bounded_embedding} follows from
$$
\int|\psi(x)|^2\,d\mu(x) = \int|\<\psi,\Phi(x)\>|^2\,d\mu(x)\,\le\,\|\psi\|^2\int\vphi(x)\,d\mu(x) = \|\psi\|^2\|\vphi\|_1.
$$
Assume that (A2) holds and that $\psi\in L^2_\mu(\calX)$ is such that $\<\psi,\Phi(x)\>_\mu = 0$ for all $x\in \calX$. Then
$$
0 = \int\<\psi,\Phi(x)\>_\mu\psi(x)\,d\mu(x) = \int\!\int k(x,y)\psi(x)\psi(y)\,d\mu(x)\,d\mu(y).
$$
Hence, $\psi=0$ by (A2). Conversely, assume that $\bH$ is dense in $L^2_\mu(\calX)$. Let $\psi\in L^2_\mu(\calX)$ such that we have $\int\!\int k(x,y)\psi(x)\psi(y)\,d\mu(x)\,d\mu(y) = 0$. Since the integrand equals $\<\psi(x)\Phi(x),\psi(y)\Phi(y)\>$ and the corresponding integral $\int\psi(x)\Phi(x)\,d\mu(x)$ exists by \eqref{e:calE_bounded}, we obtain $\int\psi(x)\Phi(x)\,d\mu(x) = 0_\bH$. This implies that $\<\psi,\Phi(y)\>_\mu = \int\psi(x)k(x,y)\,d\mu(x) = 0$ for each $y\in \calX$. Hence, $\<\psi,\phi\>_\mu=0$ for each $\phi\in\calH := \linspan\{\Phi(x) : x\in \calX\}$. Now, let $\phi\in\bH$. Then there exists a sequence $(\phi_n)\subset\calH$ such that $\|\phi_n-\phi\|\to 0$ as $n\to\infty$. Therefore,
$$
|\<\psi,\phi\>_\mu| = |\<\psi,\phi - \phi_n\>_\mu|\,\le\,\|\psi\|_\mu\|\phi-\phi_n\|_\mu\,\le\,\|\psi\|_\mu\sqrt{\|\vphi\|_1}\|\phi-\phi_n\|.
$$
Hence, $\<\psi,\phi\>_\mu=0$, and the density of $\bH$ in $L^2_\mu(\calX)$ implies $\psi=0$.
\end{proof}

\begin{proof}[Proof of Lemma \ref{l:ECH}]
(a) For $\psi\in L^2_\mu(\calX)$ we have
\begin{align*}
\|\calE\psi\|^2 = \int\int\psi(x)\psi(y)\<\Phi(x),\Phi(y)\>\,d\mu(x)\,d\mu(y) = \int\int k(x,y)\psi(x)\psi(y)\,d\mu(x)\,d\mu(y).
\end{align*}
Hence, the injectivity of $\calE$ follows from (A2). If $(e_i)$ is an orthonormal basis of $\bH$, then
\begin{align*}
\sum_i\|\calE^*e_i\|_\mu^2
&= \sum_i\|e_i\|_\mu^2 = \sum_i\int|e_i(x)|^2\,d\mu(x) = \sum_i\int|\<\Phi(x),e_i\>|^2\,d\mu(x) = \int\|\Phi(x)\|^2\,d\mu(x).
\end{align*}
The claim is now a consequence of $\|\Phi(x)\|^2 = \vphi(x)$.

(b) By Lemma \ref{l:basics0}, $\bH$ is dense in $L^2_\mu(\calX)$. Moreover, $\calE^*$ is compact by (a) and Schauder's theorem \cite[Theorem 4.19]{rfa}.

(c) This follows from (a) and $\ker C_\bH = \ker\calE\calE^* = \ker\calE^* = \{0\}$ by (A3).
\end{proof}

\begin{proof}[Proof of Theorem \ref{t:mercer}]
By Lemma \ref{l:ECH}, the operator $\calE^*\calE\in\calB(L^2_\mu(\calX))$ is a positive self-adjoint trace-class operator. Hence, by the well known spectral theory of compact operators (see, e.g., \cite{gk}) there exists an orthonormal basis $(e_j)_{j=1}^\infty$ of $L^2_\mu(\calX)$ consisting of eigenfunctions of $\calE^*\calE$ corresponding to a summable sequence $(\la_j)_{j=1}^\infty$ of strictly positive eigenvalues. Since $\calE^*\psi = \psi$ for $\psi\in\bH$, we have $\calE e_j = \la_j e_j$ and thus $e_j\in\bH$ for all $j$ and $C_\bH e_j = \calE\calE^* e_j = \calE e_j = \la_j e_j$. Moreover, $\<f_i,f_j\> = \sqrt{\la_j/\la_i}\<\calE e_i,e_j\> = \sqrt{\la_j/\la_i}\<e_i,e_j\>_\mu = \delta_{ij}$ by \eqref{e:adj} so that the $f_j$ indeed form an orthonormal system in $\bH$. The completeness of $(f_j)$ in $\bH$ follows from the injectivity of $\calE$. Finally, $\sum_{j=1}^\infty\la_j = \Tr C_\bH = \|\vphi\|_1$ and
$$
k(x,y) = \<\Phi(x),\Phi(y)\> = \sum_j\<\Phi(x),f_j\>\<f_j,\Phi(y)\> = \sum_jf_j(x)f_j(y),
$$
which completes the proof.
\end{proof}

\begin{proof}[Proof of Proposition \ref{p:Kbounded}]
Let $\psi\in B(\calX)$. For $p=\infty$ we have $|(K^t\psi)(x)| = |\bE^x[\psi(X_t)]|\le\bE^x[|\psi(X_t)|]\le \|\psi\|_\infty$. If $p<\infty$, by Jensen's inequality, for every convex $\phi : \R\to\R$ we have $\phi\circ K^t\psi\le K^t(\phi\circ\psi)$ and thus $|K^t\psi|^p\le K^t|\psi|^p$, which, by invariance of $\mu$, leads to
$$
\|K^t\psi\|_p^p = \int|K^t\psi|^p\,d\mu\le \int K^t|\psi|^p\,d\mu = \int|\psi|^p\,d\mu = \|\psi\|_p^p.
$$
The claim now follows by density of $B(\calX)$ in $L^p_\mu(\calX)$.
\end{proof}

\begin{proof}[Proof of Proposition \ref{p:semigroup}]
Let $\psi\in C_b(\calX)$ and fix $x\in \calX$. Denote the stochastic solution process of the SDE \eqref{e:SDE} with initial value $x$ by $X_t^x$. Since $X_t^x(\omega)$ is continuous in $t$ for $\bP$-a.e.\ $\omega\in\Omega$ (see \cite[Theorem 5.2.1]{oe}), $\psi(X_t^x(\omega))\to\psi(X_0^x(\omega)) = \psi(x)$ as $t\to 0$ for $\bP$-a.e. $\omega\in\Omega$. Hence, by dominated convergence,
$$
K^t\psi(x) = \bE[\psi(X_t^x)] = \int\psi(X_t^x(\omega))\,d\bP(\omega)\,\to\,\psi(x)
$$
as $t\to 0$. It now follows from Proposition \ref{p:Kbounded} and, again, dominated convergence that $\|K^t\psi - \psi\|_p\to 0$ as $t\to 0$. If $\psi\in L_\mu^p(\calX)$ and $\veps>0$, there exists $\eta\in C_b(\calX)$ such that $\|\psi-\eta\|_p < \veps/3$. Choose $\delta>0$ such that $\|K^t\eta-\eta\|_p < \veps/3$ for $t < \delta$. Then
$$
\|K^t\psi - \psi\|_p\le \|K^t(\psi - \eta)\|_p + \|K^t\eta-\eta\|_p + \|\eta-\psi\|_p < \veps
$$
for $t < \delta$, which proves the claim.
\end{proof}


\section{Some facts from spectral theory}
In this section, let $\calH$ be a Hilbert space. If $P$ is an orthogonal projection in $\calH$, we set $P^\perp = I - P$. For $v\in\calH$, $\|v\|=1$, denote by $P_v$ the rank-one orthogonal projection onto $\linspan\{v\}$.

We say that a linear operator on $\calH$ is {\em non-negative} if it is self-adjoint and its spectrum is contained in $[0,\infty)$. For a non-negative compact operator $T$ on $\calH$ we denote by $\la_1(T)\ge\la_2(T)\ge\dots$ the eigenvalues of $T$ in descending order (counting multiplicities). We set $\la_j(T) = 0$ if $j>\rank(T)$. Moreover, if $T$ has only simple eigenvalues\footnote{i.e., $\dim\ker(T-\la) = 1$ for each eigenvalue $\la$ of $T$}, we let $P_j(T)$ denote the orthogonal projection onto the eigenspace $\ker(T - \la_j(T))$ and $Q_n(T) = \sum_{j=1}^n P_j(T)$ the spectral projection corresponding to the $n$ largest eigenvalues of $T$.

\begin{thm}[{\cite[Cor.\ II.2.3]{gk}}]\label{t:eigs}
If $T$ and $\wh T$ are two non-negative compact operators on $\calH$, then for all $j\in\N$,
$$
|\la_j(T) - \la_j(\wh T)|\,\le\,\|T-\wh T\|.
$$
\end{thm}

\begin{lem}\label{l:proj_diff}
For $v,w\in\calH$ with $\|v\|=\|w\|=1$ we have
\begin{equation}\label{e:proj}
\|P_v - P_w\| = \|P_w^\perp P_v\| = \sqrt{1 - |\<v,w\>|^2}.
\end{equation}
\end{lem}
\begin{proof}
First of all, the second equation in \eqref{e:proj} is clear, since
$$
\|P_w^\perp P_vf\|^2 = \|\<f,v\>P_w^\perp v\|^2 = |\<f,v\>|^2(1 - \|P_wv\|^2) = |\<f,v\>|^2(1 - |\<v,w\>|^2).
$$
Second, if $P_{v,w}$ denotes the orthogonal projection onto $\calH_{v,w} := \linspan\{v,w\}$, we have
$$
\|P_v - P_w\| = \|(P_v - P_w)P_{v,w}\| = \|(P_v - P_w)|_{\calH_{v,w}}\| = \sup_{x\in\calH_{v,w},\,\|x\|=1}\|(P_v-P_w)x\|,
$$
which is a two-dimensional problem in $\calH_{v,w}$. Now, if $x\in\calH_{v,w}$, $\|x\|=1$, we write $x = av+bw$ and obtain $a^2+2ab\gamma + b^2 = 1$, where $\gamma = \<v,w\>$. Moreover, $\<x,v\> = a+b\gamma$, $\<x,w\>=a\gamma + b$ and so
\begin{align*}
\|(P_v - P_w)x\|^2
&= \|\<x,v\>v - \<x,w\>w\|^2 = \|(a+b\gamma)v - (a\gamma+b)w\|^2\\
&= (a+b\gamma)^2 - 2(a+b\gamma)(a\gamma+b)\gamma + (a\gamma+b)^2\\
&= a^2+2ab\gamma + b^2\gamma^2 - 2\gamma(a^2\gamma + ab\gamma^2 + ab + b^2\gamma) + a^2\gamma^2 + 2ab\gamma + b^2\\
&= (1-\gamma^2)a^2 + 2ab\gamma - 2ab\gamma^3 + b^2(1-\gamma^2)\\
&= (1-\gamma^2)(a^2+b^2+2ab\gamma)\\
&= 1 - |\<v,w\>|^2.
\end{align*}
Hence, the objective function is constant on $\{x\in\calH_{v,w} : \|x\|=1\}$ and \eqref{e:proj} is proved.
\end{proof}

The next theorem is a variant of the Davis-Kahan $\sin(\Theta)$ theorem (cf.\ \cite{yws}).

\begin{thm}\label{t:dk_special}
Let $T$ and $\wh T$ be non-negative Hilbert-Schmidt operators on $\calH$, let $n\in\N$, assume that the largest $n+1$ eigenvalues of $T$ are simple, and set
$$
\delta = \min_{j=1,\ldots, n}\frac{\la_j(T) - \la_{j+1}(T)}2.
$$
If $\|T - \wh T\|_{HS} < \delta$, then for $j=1,\ldots,n$ we have
$$
\|P_j(T) - P_j(\wh T)\|\le\frac{\|T - \wh T\|}{\delta}.
$$
\end{thm}
\begin{proof}
For $j\in\N$ put $\la_j = \la_j(T)$, $P_j = P_j(T)$, $\wh\la_j = \la_j(\wh T)$, and $\wh P_j = P_j(\wh T)$. By Theorem \ref{t:eigs}, we have $|\la_j - \wh\la_j|\le\|T-\wh T\|_{HS} < \delta$ for all $j$, hence $\wh\la_j$ is contained in the interval $I_j = (\la_j-\delta,\la_j+\delta)$ for $j=1,\ldots,n+1$. By assumption, $\sup I_{j+1}\le\inf I_j$ for $j=1,\ldots,n$. In particular, the intervals $I_1,\ldots,I_{n+1}$ are pairwise disjoint.

Now, let $j\in\{1,\ldots,n\}$. Then for $k\in\N\setminus\{j\}$ we have $|\wh\la_k - \la_j| > \delta$. Therefore, we have $\dist(\la_j,\sigma(\wh T)\backslash\{\wh\la_j\})\ge\delta$ and thus, for $f\in\wh P_j^\perp\calH$,
$$
\|(\wh T-\la_j)f\|\,\ge\,\dist\big(\la_j,\sigma(\wh T|_{\wh P_j^\perp\calH})\big)\|f\| = \dist(\la_j,\sigma(\wh T)\backslash\{\wh\la_j\})\|f\|\,\ge\,\delta\|f\|.
$$
As $TP_j = \la_jP_j$ and $\wh P_j^\perp\wh T = \wh T\wh P_j^\perp$, we obtain
$$
\|T - \wh T\|\,\ge\,\|\wh P_j^\perp(\wh T-T)P_j\| = \|\wh P_j^\perp\wh TP_j - \wh P_j^\perp TP_j\| = \|(\wh T - \la_j)\wh P_j^\perp P_j\|\,\ge\,\delta\|\wh P_j^\perp P_j\|.
$$
The claim now follows from Lemma \ref{l:proj_diff}.
\end{proof}

\section{The adjoint of the Koopman operator}\label{a:perron}
For $p\in [1,\infty)$ let $K_p^t : L^p_\mu\to L^p_\mu$ denote the Koopman operator on $L^p_\mu$. If $p=2$, we still write $K_2^t = K^t$. For a Borel set $A\in\frakB(\R^d)$ and $f\in L^q_\mu(\R^d)$ we have
\[
\int_A (K_p^t)^*f\,d\mu = \<(K_p^t)^*f,\one_A\>_{L_\mu^q,L_\mu^p} = \<f,K_p^t\one_A\>_{L_\mu^q,L_\mu^p} = \int\rho_t(x,A)f(x)\,d\mu(x).
\]
Since the right-hand side also makes sense for $f\in L^1_\mu$, the operator $(K^t)^*$ extends to a bounded operator $P^t$ on $L^1_\mu$.  From the above defining identity, it is readily seen that $P^t$ is a Markov operator, i.e., $P^t\one = \one$ and $P^tf\ge 0$ if $f\ge 0$.

Let $f$ be a simple function, i.e., $f = \sum_{i=1}^n a_i\one_{A_i}$, where the $A_i$ are mutually disjoint and $\bigcup_{i=1}^nA_i = \R^d$. Then $\sum_{i=1}^nP^t\one_{A_i} = (K^t)^*\one = \one$, hence, by convexity of $z\mapsto z^2$,
\[
P^tf^2(x) = \sum_{i=1}^na_i^2P^t\one_{A_i}(x)\,\ge\,\Big(\sum_{i=1}^na_iP^t\one_{A_i}(x)\Big)^2 = [P^tf(x)]^2,
\]
and therefore $P^tf^2\ge [P^tf]^2$. Similarly, $P^t|f|\ge|P^tf|$, which shows $\|P^t\|\le 1$. If $f\in L^2_\mu$, let $(f_n)$ be a sequence of simple functions such that $\|f_n-f\|_\mu\to 0$ as $n\to\infty$. Then
\begin{align*}
\int_A ([(K^t)^*f]^2 - (K^t)^*f^2)\,d\mu
&\le\int_A ([P^tf]^2 - [P^tf_n]^2)\,d\mu + \int_A (P^tf_n^2 - P^tf^2)\,d\mu\\
&\le \int P^t(f-f_n)P^t(f+f_n)\,d\mu + \|P^t(f_n^2-f^2)\|_{L^1_\mu}\\
&\le \|(K^t)^*(f_n-f)\|_\mu\|(K^t)^*(f_n+f)\|_\mu + \|f_n^2-f^2\|_{L^1_\mu}\\
&\le 2\|f_n-f\|_\mu\|f_n+f\|_\mu
\end{align*}
for every $A\in\frakB(\R^d)$. This proves $[(K^t)^*f]^2\le (K^t)^*f^2$ $\mu$-a.e.\ for all $f\in L^2_\mu(\R^d)$.

\begin{lem}\label{l:geht}
Let $f,g,h\in L^2_\mu$ such that $g^2(K^t)^*f^2\in L^1_\mu$. Then $f\cdot K_1^t(gh)\in L^1_\mu$ and
\[
\int f\cdot K_1^t(gh)\,d\mu = \int [g\cdot(K^t)^*f]\cdot h\,d\mu.
\]
\end{lem}
\begin{proof}
Let $(f_n)\subset L^\infty_\mu$ be a sequence of simple functions such that $f_n\to f$ and $|f_n|\nearrow|f|$ pointwise $\mu$-a.e.\ as $n\to\infty$. Then
\[
\int|(K^t)^*(f_n-f)|\,d\mu\,\le\,\int(K^t)^*|f_n-f|\,d\mu = \int|f_n-f|\,d\mu,
\]
which converges to zero as $n\to\infty$ by majorized convergence ($|f_n-f|\le 2|f|$). Hence, there exists a subsequence $(f_{n_k})$ such that $(K^t)^*f_{n_k}\to (K^t)^*f$ $\mu$-a.e.\ as $k\to\infty$. WLOG, we may therefore assume that $(K^t)^*f_n\to (K^t)^*f$ $\mu$-a.e.\ as $n\to\infty$. By monotone convergence,
\begin{align*}
\int |f||K_1^t(gh)|\,d\mu
&= \lim_{n\to\infty}\int |f_n||K_1^t(gh)|\,d\mu\le \limsup_{n\to\infty}\int |f_n|\cdot K_1^t[|gh|]\,d\mu\\
&= \limsup_{n\to\infty}\int (K_1^t)^*|f_n|\cdot |gh|\,d\mu \,\le\, \int |g|(K^t)^*|f|\cdot |h|\,d\mu,
\end{align*}
which is a finite number. Hence, indeed, $f\cdot K_1^t(gh)\in L^1_\mu$ and by majorized convergence,
\begin{align*}
\int f\cdot K_1^t(gh)\,d\mu
&= \lim_{n\to\infty}\int f_n\cdot K_1^t(gh)\,d\mu = \lim_{n\to\infty}\int [g\cdot(K_1^t)^*f_n]\cdot h\,d\mu = \int [g\cdot(K^t)^*f]\cdot h\,d\mu,
\end{align*}
as claimed.
\end{proof}

\section{Ergodicity and the Koopman semigroup}\label{a:ergodicity}
In this section, we prove the following proposition on the spectral properties of the generator $\calL$ under the ergodicity assumption.

\begin{prop}\label{p:L_and_ergo}
Assume that the invariant measure $\mu$ is ergodic. Then $\ker(I - K^t) = \linspan\{\one\}$ for all $t>0$. In particular, $\ker\calL = \linspan\{\one\}$ and $\ker(\calL - i\omega I) = \{0\}$ for $\omega\in\R\backslash\{0\}$.
\end{prop}
\begin{proof}
Concerning the ``in particular''-part, we only mention that $\calL\psi=0$ implies $K^t\psi=\psi$ for all $t\ge 0$ and that $\calL\psi = i\omega\psi$, $\omega\in \R \setminus \{0\}$, implies $K^{2\pi/\omega}\psi = \psi$. So, let us show that $K^t\psi = \psi$ for some $t>0$ and $\psi\in L^2_\mu(\calX)$ is only possible for constant $\psi$. For this, we consider the Markov process $(X_{nt})_{n=0}^\infty$. For convenience, we assume w.l.o.g.\ 
that $t=1$. By invariance of $\mu$, the process $(X_n)_{n=0}^\infty$ is stationary, i.e., $(X_n)_{n=0}^\infty$ and $(X_{n+1})_{n=0}^\infty$ are equally distributed as $\calX^{\N_0}$-valued random variables. According to \cite[Lemma 9.2]{kall} there exist $\calX$-valued random variables $X_{-k}$, $k\in\N$, such that $X := (X_n)_{n\in\Z}$ is also stationary. By $P_\mu$ denote the law of the $\calX^\Z$-valued random variable $X$.

On $S := \calX^\Z$ define the left shift $T : S\to S$ by $T(x_n)_{n\in\Z} := (x_{n+1})_{n\in\Z}$. Stationarity of $X$ means that also $TX\sim P_\mu$.

A set $\calA\in\calB_\calX^\Z := \bigotimes_{k\in\Z}\calB_\calX$ is called shift-invariant if $T^{-1}\calA = \calA$. It is easy to see that the set of shift-invariant sets forms a sub-$\sigma$-algebra $\calI$ of $\calB_\calX^\Z$. Now, by \cite[Corollary 5.11]{h} and the ergodicity of $\mu$ we have $P_\mu(\calA)\in\{0,1\}$ for any $\calA\in\calI$. Now, Birkhoff's Ergodic Theorem \cite[Theorem 9.6]{kall} states that
\begin{equation}\label{e:birkhoff}
\lim_{n\to\infty}\frac 1n\sum_{k=0}^{n-1}f(T^kX) = \bE\big[f(X)|X^{-1}\calI\big]
\end{equation}
almost surely and in $L^1(\Omega)$ for any $f\in L^1(S)$. Given $\psi\in L^1_\mu(\calX)$, let us apply this theorem to the function $f = \psi\circ\pi_0$, where the projection $\pi_0 : S\to X$ is defined by $\pi_0(x_n)_{n\in\Z} = x_0$. First of all,
$$
\int|f|\,dP_\mu = \int|\psi(x_0)|\,dP_\mu((x_n)_{n\in\Z}) = \int|\psi(x)|\,d\mu(x) < \infty
$$
as $P_\mu\circ\pi_0^{-1} = \mu$. Hence, we have $f\in L^1(S)$. Furthermore, we compute $f(T^kX) = \psi(\pi_0(T^kX)) = \psi(X_k)$. For $\calA\in\calI$ we have $\bP(X^{-1}\calA) = P_\mu(\calA)\in\{0,1\}$. Thus, we obtain
$$
\lim_{n\to\infty}\frac 1n\sum_{k=0}^{n-1}\psi(X_k) = \bE[f(X)] = \int f\,dP_\mu = \int\psi\circ\pi_0\,dP_\mu = \int\psi\,d\mu
$$
almost surely and in $L^1(\Omega)$. Therefore, if $\psi\in L^2_\mu(\calX)$ such that $K^t\psi = \psi$, then $K^{kt}\psi = \psi$ for all $k\in\N_0$, hence for $\mu$-a.e.\ $x\in X$ we have
\begin{align*}
\psi(x)
&= \frac 1n\sum_{k=0}^{n-1}K^{kt}\psi(x) = \frac 1n\sum_{k=0}^{n-1}\bE[\psi(X_{kt})|X_0=x]\\
&= \bE\left[\frac 1n\sum_{k=0}^{n-1}\psi(X_{kt})\,\Bigg|\,X_0=x\right]\,\overset{n\to\infty}{\longrightarrow}\int\psi\,d\mu.
\end{align*}
Thus, $\psi$ must indeed be ($\mu$-essentially) constant.
\end{proof}

\section{Koopman-invariance of the Gaussian RKHS in case of the Ornstein-Uhlenbeck process}\label{s:OU_invariance}
In many works in the present literature on the Koopman operator for deterministic dynamical systems in connection with kernels, it is assumed that the Koopman operator maps the RKHS boundedly (or even compactly) into itself. However, as it was proved in \cite{gakr}, the Koopman operator of a discrete-time system on $\R^n$ is invariant under the radial basis function RKHS if and only if the dynamics are affine.

In the following, we show that the situation is essentially different for stochastic systems in that we prove that the RBF RKHS is invariant under the Koopman operator associated with the OU process.

The Ornstein-Uhlenbeck (OU) process on $X = \R$ is the solution of the SDE $dX_t = -\alpha X_t\,dt + dW_t$. The invariant measure $\mu$ and the Markov transition kernel $\rho_t$, $t>0$, are known and given by
\begin{align*}
d\mu(x) = \sqrt{\frac{\alpha}{\pi}}\cdot e^{-\alpha x^2}\,dx
\qquad\text{and}\qquad
\rho_t(x,dy) = \sqrt{\frac{c_t}{\pi}}\cdot\exp\Big[-c_t(y - e^{-\alpha t}x)^2\Big]\,dy,
\end{align*}
where
$$
c_t = \frac{\alpha}{1 - e^{-2\alpha t}}.
$$
The Koopman operators $K^t$ are self-adjoint in $L^2_\mu(\R)$.

We consider the Gaussian radial basis function (RBF) kernel with bandwidth $\sigma>0$, i.e.,
\begin{equation*}
    k^\sigma(x, y) = \exp\left[-\frac{(x-y)^2}{\sigma^2}\right].
\end{equation*}
By $\bH_\sigma$ we denote the RKHS generated by the kernel $k^\sigma$. Hilbert space norm and scalar product on $\bH_\sigma$ will be denoted by $\|\cdot\|_\sigma$ and $\<\cdot\,,\,\cdot\>_\sigma$, respectively. For $y\in\R$ and a kernel $k$ on $\R$ set $k_y(x) := k(x,y)$. For two positive definite kernels on $\R$ we write $k_1\preceq k_2$ if
$$
\sum_{i,j=1}^n\alpha_i\alpha_j k_1(x_i,x_j)\le \sum_{i,j=1}^n\alpha_i\alpha_j k_2(x_i,x_j)
$$
for any choice of $n\in\N$ and $\alpha_j,x_j\in\R$, $j=1,\ldots,n$. We also write $V\lhook\joinrel\xrightarrow{\hspace{.1cm}c\hspace{.1cm}}W$ for two normed vector spaces $V$ and $W$ if $V\subset W$ is continuously embedded in $W$.

\begin{lem}
Let $0 < \sigma_1 < \sigma_2$. Then we have  $\bH_{\sigma_2}\!\lhook\joinrel\xrightarrow{\hspace{.1cm}c\hspace{.1cm}}\bH_{\sigma_1}$ with $\|\psi\|_{\sigma_1}\le\sqrt{\frac{\sigma_2}{\sigma_1}}\|\psi\|_{\sigma_2}$ for $\psi\in\bH_{\sigma_2}$ and $k^{\sigma_2}\preceq \frac{\sigma_2}{\sigma_1}k^{\sigma_1}$.
\end{lem}
\begin{proof}
The first claim is Corollary 6 in \cite{shs}. The second follows from Aronszajn's inclusion theorem \cite[Theorem 5.1]{pr}.
\end{proof}

\begin{prop}
For each $t\ge 0$ and all $\alpha,\sigma > 0$ we have $K^t\bH_\sigma\subset\bH_\sigma$ with
$$
\big\|K^t\big\|_{\bH_\sigma\to\bH_\sigma}\le e^{\frac\alpha 2t}.
$$
\end{prop}
\begin{proof}
For $t=0$ the claim is obviously true. Hence, suppose that $t>0$. For $z\in\R$ let us compute $K^tk^\sigma_z$. For this, we define
$$
\tau = \sqrt{\frac{c_t\sigma^2}{1+c_t\sigma^2}},
\qquad\text{and}\qquad
\nu = \frac{e^{\alpha t}}{\tau}\sigma\,>\,\frac\sigma\tau > \sigma.
$$
Since for $\sigma_1,\sigma_2>0$ and $z,w\in\R$ we have
\[
\int_{-\infty}^\infty k_z^{\sigma_1}(x)\,k_w^{\sigma_2}(x)\,dx = \sqrt{\pi}\cdot\frac{\sigma_1\sigma_2}{\sqrt{\sigma_1^2 + \sigma_2^2}}\cdot\exp\left(-\frac{(z-w)^2}{\sigma_1^2+\sigma_2^2}\right),
\]
with $\sigma_t = 1/\sqrt{c_t}$ we obtain
\begin{align*}
(K^tk_z^\sigma)(x)
&= \int_{-\infty}^\infty k_z^\sigma(y)\,\rho_t(x,dy) = \sqrt{\frac{c_t}{\pi}}\int_{-\infty}^\infty k_z^\sigma(y)\exp\Big[-c_t(y - e^{-\alpha t}x)^2\Big]\,dy\\
&= \sqrt{\frac{c_t}{\pi}}\int_{-\infty}^\infty k_z^\sigma(y)k_{e^{-\alpha t}x}^{\sigma_t}(y)\,dy = \sqrt{c_t}\cdot\frac{\sigma\sigma_t}{\sqrt{\sigma^2 + \sigma_t^2}}\cdot \exp\left(-\frac{(z-e^{-\alpha t}x)^2}{\sigma^2+\sigma_t^2}\right)\\
&= \frac{\sigma}{\sqrt{\sigma^2 + 1/c_t}}\cdot\exp\left(-\frac{(e^{\alpha t}z-x)^2}{e^{2\alpha t}(\sigma^2+1/c_t)}\right)
= \tau\cdot k^\nu_{e^{\alpha t}z}(x).
\end{align*}
That is,
$$
K^tk_z^\sigma = \tau\cdot k^\nu_{e^{\alpha t}z}\in\bH_\nu.
$$
Note that
$$
k^\nu(e^{\alpha t}x,e^{\alpha t}y) = \exp\left(-\frac{e^{2\alpha t}(x-y)^2}{\nu^2}\right) = \exp\left(-\frac{\tau^2(x-y)^2}{\sigma^2}\right) = k^{\sigma/\tau}(x,y).
$$
Now, let $n\in\N$ and $\alpha_j,x_j\in\R$, $j=1,\ldots,n$, be arbitrary and let $\psi = \sum_{j=1}^n\alpha_jk^\sigma_{x_j}$. Then
\begin{align*}
\|K^t\psi\|_\sigma^2
&\le\frac{\nu}{\sigma}\|K^t\psi\|_\nu^2
= \frac{\nu}{\sigma}\Bigg\|\sum_{j=1}^n\alpha_jK^tk^\sigma_{x_j}\Bigg\|_\nu^2 = \frac{\nu\tau^2}{\sigma}\Bigg\|\sum_{j=1}^n\alpha_jk^\nu_{e^{\alpha t}x_j}\Bigg\|_\nu^2\\
&= \frac{\nu\tau^2}{\sigma}\sum_{i,j=1}^n\alpha_i\alpha_jk^\nu(e^{\alpha t}x_i,e^{\alpha t}x_j) = \frac{\nu\tau^2}{\sigma}\sum_{i,j=1}^n\alpha_i\alpha_jk^{\sigma/\tau}(x_i,x_j)\\
&\le \frac{\nu\tau^2}{\sigma}\cdot\frac 1\tau\sum_{i,j=1}^n\alpha_i\alpha_jk^{\sigma}(x_i,x_j) = \frac{\nu\tau}\sigma\|\psi\|_\sigma^2 = e^{\alpha t}\|\psi\|_\sigma^2.
\end{align*}
This shows that $K^t$ maps $\bH_{0,\sigma} = \linspan\{k^\sigma_x : x\in\R\}\subset\bH_\sigma$ boundedly into $\bH_\sigma$. Since $\bH_{0,\sigma}$ is dense in $\bH_\sigma$, it follows that $K^t|_{\bH_{0,\sigma}}$ extends to a bounded operator $T$ in $\bH_\sigma$. In order to see that $T\psi = K^t\psi$ for $\psi\in\bH_\sigma$, let $(\psi_n)\subset\bH_{0,\sigma}$ such that $\psi_n\to\psi$ in $\bH_\sigma$. Then $K^t\psi_n = T\psi_n\to T\psi$ in $\bH_\sigma$. Since $\bH_\sigma\lhook\joinrel\xrightarrow{\hspace{.1cm}c\hspace{.1cm}}L^2_\mu(\R)$, we have $\psi_n\to\psi$ in $L^2_\mu(\R)$ and thus $K^t\psi_n\to K^t\psi$ in $L^2_\mu(\R)$. Also, $K^t\psi_n\to T\psi$ in $L^2_\mu(\R)$. Hence, $K^t\psi=T\psi$ $\mu$-a.e.\ on $\R$. But as both $K^t\psi$ and $T\psi$ are continuous and $\mu$ is absolutely continuous w.r.t.\ Lebesgue measure with a positive density, we conclude that $K^t\psi = T\psi\in\bH_\sigma$.
\end{proof}

\begin{cor}
The restriction $K^t_\sigma := K^t|_{\bH_\sigma}\in L(\bH_\sigma)$ does not have eigenvalues and is thus not compact. Moreover, $K_\sigma^t$ is not self-adjoint.
\end{cor}
\begin{proof}
It is well known that $K^t$ (on $L^2_\mu$) has the eigenvalues $\la_j = e^{-\alpha jt}$, $j\in\N$, and that the corresponding eigenfunctions $\psi_j$ are polynomials. Let now $K^t_\sigma\psi=\la\psi$ for some $\psi\in\bH_\sigma$ and $\la\in\C$. Then also $K^t\psi = \la\psi$ and so $\la=\la_j$ and $\psi\in\linspan\{\psi_j\}$ for some $j$. Hence, $\psi$ is a polynomial. Since for all $x\in\R$ we have $|\psi(x)| = |\<\psi,k_x^\sigma\>_\sigma|\le \|\psi\|_\sigma\|k_x^\sigma\|_\sigma = \|\psi\|_\sigma$, it follows that $\psi$ is bounded, which is only possible for $\psi\in\linspan\{\one\}$. But $\one\notin\bH_\sigma$ by \cite[Corollary 5]{shs}. Hence, $\psi=0$. The non-selfadjointness of $K_\sigma^t$ can be easily seen from computing $\<K_\sigma^tk_z^\sigma,k_w^\sigma\>_\sigma$ and $\<k_z^\sigma,K_\sigma^tk_w^\sigma\>_\sigma$ for $z\neq w$.
\end{proof}

\bigskip
\section*{Author Affiliations}

\begin{thebibliography}{AA}
\bibitem{a}
N. Aronszajn,
Theory of reproducing kernels,
Trans. Amer. Math. Soc. 68 (1950), 337--404.



\bibitem{Beck57}
A. Beck, and J. T. Schwartz,
A vector-valued random ergodic theorem,
Proc. Amer. Math. Soc. 8(6) (1957), 1049--1059.

\bibitem{bblshh}
P. Bevanda, M. Beier, A. Lederer, S. Sosnowski, E. Hüllermeier, S. Hirche,
Koopman kernel regression,
Preprint, arXiv:2305.16215.

\bibitem{bosq}
D. Bosq,
Linear Processes in Function Spaces -- Theory and Applications,
Springer-Verlag New York, Inc., 2000.

\bibitem{brunton16}
S. L. Brunton, J.L. Proctor, and J.N. Kutz,
Discovering governing equations from data by sparse identification of nonlinear dynamical systems,
Proc. Natl. Acad. Sci. 113 (2016), 3932--3937.

\bibitem{brunton22}
S. L. Brunton, M. Budisic, E. Kaiser, J. N. Kutz,
Modern Koopman theory for dynamical systems,
SIAM Rev. 64(2) (2022), 229--340.

\bibitem{bt}
A. Berlinet and C. Thomas-Agnan,
Reproducing Kernel Hilbert Spaces in Probability and Statistics,
Kluwer Academic Publishers, 2004.


\bibitem{das2021}
S. Das, D. Giannakis, and J. Slawinska,
Reproducing kernel Hilbert space compactification of unitary evolution groups,
Appl. Comput. Harmon. Anal. 54 (2021), 75--136.

\bibitem{Douglas1966}
R.G. Douglas,
On majorization, factorization, and range inclusion of operators on Hilbert space,
Proc. Amer. Math. Soc. 17 (1966), 413--415.

\bibitem{fass12}
G. E. Fasshauer and M. J. McCourt,
Stable Evaluation of Gaussian Radial Basis Function Interpolants
SIAM J. Sci. Comput. 2012 34:2, A737--A762.

\bibitem{fsg}
K. Fukumizu, L. Song, and A. Gretton,
Kernel Bayes' Rule: Bayesian inference with positive definite kernels,
J. Mach. Learn. Res. 14 (2013), 3753--3783.

\bibitem{gia19}
D. Giannakis,
Data-driven spectral decomposition and forecasting of ergodic dynamical systems,
Appl. Comput. Harmon. Anal, 47(2) (2019), 338--396.

\bibitem{gk}
I.C. Gohberg and M.G. Krein,
Introduction To The Theory of Linear Nonselfadjoint Operators,
volume 18 of Translations of Mathematical Monographs. American Mathematical Society, 1969.

\bibitem{gakr}
E. Gonzalez, M. Abudia, M. Jury, R. Kamalapurkar, and J.A. Rosenfeld,
The kernel perspective on dynamic mode decomposition, arXiv:2106.00106.

\bibitem{gumah22}
G. Gumah,
Numerical solutions of special fuzzy partial differential equations in a reproducing kernel Hilbert space,
Comput. Appl. Math., 41(2) (2022), 80.

\bibitem{gab}
G. Gumah, S. Al-Omari, and D. Baleanu,
Soft computing technique for a system of fuzzy Volterra integro-differential equations in a Hilbert space,
Appl. Num. Math. 152 (2020), 310--322.

\bibitem{gnaab}
G. Gumah, M.F.M. Naser, M. Al-Smadi, S.K.Q. Al-Omari, and D. Baleanu,
Numerical solutions of hybrid fuzzy differential equations in a Hilbert space,
Appl. Num. Math. 151 (2020), 402--412.

\bibitem{h}
M. Hairer,
Ergodic properties of Markov processes,
Lecture notes, https://www.hairer.org/notes/Markov.pdf

\bibitem{Hase21}
M. Haseli and J. Cort{\'e}s,
Learning Koopman eigenfunctions and invariant subspaces from data: Symmetric subspace decomposition,
IEEE Trans. Automat. Control  67(7) (2022), 3442-3457.

\bibitem{kall}
O. Kallenberg,
Foundations of modern probability,
Springer-Verlag, New York, 1997.

\bibitem{k}
T. Kato,
Perturbation Theory for Linear Operators,
Springer-Verlag Berlin Heidelberg, 1995.


\bibitem{kaiser21}
E. Kaiser, J. N. Kutz, S. L. Brunton,
Data-driven discovery of Koopman eigenfunctions for control,
Mach. Learn. Sci. Technol. 2 (2021), 035023.

\bibitem{klus16}
S. Klus, P. Koltai, and C. Sch{\"u}tte,
On the numerical approximation of the {P}erron--{F}robenius and {K}oopman operator,
J. Comput. Dyn. 1(3) (2016), 51--79.

\bibitem{klus18_rev}
S. Klus, F. N\"uske, P. Koltai, H. Wu, I. Kevrekidis, C. Sch\"utte, and F. No\'e,
Data-driven model reduction and transfer operator approximation,
J. Nonlinear Sci. 28(3) (2018), 985--1010.

\bibitem{klus20}
S. Klus, I. Schuster, K. Muandet,
Eigendecompositions of transfer operators in reproducing kernel Hilbert spaces,
J. Nonlinear Sci. 30(1) (2020), 283--315.

\bibitem{klus16_tensor}
S. Klus and C. Sch{\"u}tte,
Towards tensor-based methods for the numerical approximation of the {P}erron--{F}robenius and {K}oopman operator,
J. Comput. Dyn. 3(2) (2016), 139--161.

\bibitem{kss}
II. Klebanov, I. Schuster, and T.J. Sullivan,
A rigorous theory of conditional mean embeddings,
SIAM J. Math. Data Sci. 2 (2020), 583--606.

\bibitem{Koo31}
B. O. {K}oopman,
Hamiltonian Systems and Transformations in Hilbert Space,
Proc. Natl. Acad. Sci.
5 (17) (1931), 315--318.

\bibitem{KordMezi18}
M. Korda and I. Mezi{\'c},
Linear predictors for nonlinear dynamical systems: Koopman operator meets model predictive control,
Automatica 
93 (2018), 149--160.

\bibitem{KordMezi20}
M. Korda and I. Mezi{\'c},
Optimal construction of Koopman eigenfunctions for prediction and control,
IEEE Trans. Automat. Control
65 (12) (2020), 5114--5129.

\bibitem{kostic22}
V. Kostic, P. Novelli, A. Maurer, C. Ciliberto, L. Rosasco, and M. Pontil, Learning dynamical systems via Koopman operator regression in reproducing kernel Hilbert spaces,
Adv. Neural. Inf. Process. Syst. 35 (2022), 4017-4031.

\bibitem{kostic23}
V. Kostic, K. Lounici, P. Novelli, and M. Pontil,
Sharp spectral rates for Koopman operator learning,
Preprint, arXiv:2302.02004.

\bibitem{kurdila18}
A. J. Kurdila and P. Bobade,
Koopman theory and linear approximation spaces,
Preprint, arXiv:1811.10809 (2018).


\bibitem{lelievre16}
T. Leli{\`{e}}vre, and G. Stoltz,
Partial differential equations and stochastic methods in molecular dynamics,
Acta Numer. 25 (2016), 681--880.

\bibitem{lp}
G. Lumer and R.S. Phillips,
Dissipative operators in a Banach space,
Pacific J. Math. 11 (1961), 679--698.

\bibitem{lusch18}
B. Lusch, J. N. Kutz, and S. L. Brunton,
Deep learning for universal linear embeddings of nonlinear dynamics,
Nat. Commun. 9(1) (2018), 1--10.

\bibitem{m}
M. Mollenhauer,
On the Statistical Approximation of Conditional Expectation Operators,
Dissertation, Freie Universit\"at Berlin, 2021.

\bibitem{manohar18}
K. Manohar, B. W. Brunton, J. N. Kutz, S. L. Brunton,
Data-driven sparse sensor placement for reconstruction: Demonstrating the benefits of exploiting known patterns.
IEEE Control Systems Magazine, 38(3) (2018), 63--86.

\bibitem{mardt18}
A. Mardt, L. Pasquali, H. Wu, and F. No{\'e},
{VAMP}nets for deep learning of molecular kinetics,
Nat. Commun. 9, 5 (2018).

\bibitem{mauroy20}
A. Mauroy, Y. Susuki, and I. Mezić,
Koopman operator in systems and control,
Springer International Publishing, Berlin, 2020.

\bibitem{mollenhauer22}
M. Mollenhauer, S. Klus, C. Sch\"utte, and P. Koltai,
Kernel autocovariance operators of stationary processes: Estimation and convergence,
J. Mach. Learn. Res. 23(327) (2022), 1--34.

\bibitem{noe13}
F. No{\'{e}} and F. N{\"{u}}ske,
A variational approach to modeling slow processes in stochastic dynamical systems,
Multiscale Model. Simul. 11 (2013), 635--655.

\bibitem{nueske14}
F. N{\"{u}}ske, B. G. Keller, G. P{\'{e}}rez-Hern{\'{a}}ndez, A. S. J. S. Mey, and F. No{\'{e}},
Variational approach to molecular kinetics,
J. Chem. Theory Comput. 10 (2014), 1739--1752.

\bibitem{nueske21}
F. N\"uske, P. Gel\ss, S. Klus, and C. Clementi,
Tensor-based computation of metastable and coherent sets,
Physica D 427 (2021), 133018.

\bibitem{nueske23}
F. N{\"u}ske, S. Peitz, F. Philipp, M. Schaller, and K. Worthmann,
Finite-data error bounds for Koopman-based prediction and control, 
J. Nonlinear Sci. 33 (2023), 1--34.

\bibitem{oe}
B. \O ksendal,
Stochastic Differential Equations, An Introduction with Applications,
Fifth Edition, Corrected Printing, Springer-Verlag Heidelberg New York, 2000.

\bibitem{pr}
V.I. Paulsen and M. Raghupathi,
An introduction to the theory of reproducing kernel Hilbert spaces,
Cambridge Studies in Advanced Mathematics 152, Cambridge University Press, 2016.

\bibitem{pav14}
G. A. Pavliotis,
Stochastic processes and applications: diffusion processes, the Fokker-Planck and Langevin equations, 
Texts in Applied Mathematics, vol. 60, Springer, 2014.

\bibitem{peitz_klus19}
S. Peitz, and S. Klus,
Koopman operator-based model reduction for switched-system control of PDEs, 
Automatica 106 (2019), 184--191.

\bibitem{pi}
I. Pinelis,
Optimum bounds for the distributions of martingales in Banach spaces,
The Annals of Probability 22 (1994), 1679--1706.

\bibitem{prinz2011}
J. H. Prinz, H. Wu, M. Sarich, B. G. Keller, M. Senne, M. Held, J. D. Chodera, C. Sch{\"u}tte, and F. No{\'e},
Markov models of molecular kinetics: Generation and validation,
J. Chem. Phys. 17(134) (2011), 174105.

\bibitem{rn}
F. Riesz and B. Nagy,
Functional Analysis,
Blackie \& Son Ltd., Glasgow, Bombay, Toronto, 1955.

\bibitem{rhomari}
N. Rhomari,
Approximation et in\'egalit\'es exponentielles pour les sommes des vecteurs al\'eatoires d\'ependants, Comptes rendus de l'Acad\'emie des science, S\'erie 1, 334 (2002), 149--154.

\bibitem{rfa}
W. Rudin,
Functional Analysis,
Second edition, McGraw-Hill, Inc., 1991.

\bibitem{rrca}
W. Rudin,
Real and Complex Analysis,
Third edition, McGraw-Hill, Inc., 1987.


\bibitem{SchaWort22}
M. Schaller, K. Worthmann, F. Philipp, S. Peitz, and F. N{\"u}ske,
Towards reliable data-based optimal and predictive control using extended DMD,
IFAC PapersOnLine, to appear, arXiv preprint arXiv:2202.09084.


\bibitem{schuette99}
C. Sch{\"{u}}tte, A. Fischer, W. Huisinga, and P. Deuflhard,
A direct approach to conformational dynamics based on hybrid {M}onte {C}arlo,
J. Comput. Phys. 1(151) (1999), 146--168.

\bibitem{sgss}
A. Smola, A. Gretton, L. Song, and B. Sch\"olkopf,
A Hilbert Space Embedding for Distributions,
in: M. Hutter, R.A. Servedio, and E. Takimoto (Eds.): ALT 2007, Lecture Notes in Computer Science, vol. 4754, pp. 13--31, 2007.
Springer, Berlin, Heidelberg.

\bibitem{sc}
I. Steinwart and A. Christmann,
Support Vector Machines,
Springer Science+Business Media, LLC, 2008.

\bibitem{shs}
I. Steinwart, D. Hush, and C. Scovel,
An explicit description of the Reproducing Kernel Hilbert Spaces of Gaussian RBF kernels,
IEEE Transactions on Information Theory 52 (2006), 4635--4643.

\bibitem{sfl}
B.K. Sriperumbudur, K. Fukumizu, and G.R.G. Lanckriet,
Universality, characteristic kernels and RKHS embedding of measures,
J. Mach. Learn. Res. 12 (2011), 2389--2410.

\bibitem{williams15}
M. O. Williams, I. G. Kevrekidis, and C. W. Rowley,
A data-driven approximation of the {K}oopman operator: {E}xtending dynamic mode decomposition,
J. Nonlinear Sci. 25(6) (2015), 1307--1346

\bibitem{williams15_kernel}
M. O. Williams, C. W. Rowley, and I Kevrekidis,
A kernel-based method for data-driven {K}oopman spectral analysis,
J. Comput. Dyn. 2(2) (2015), 247--265.

\bibitem{wu20}
H. Wu, and F. No{\'e},
Variational approach for learning Markov processes from time series data,
J. Nonlinear Sci. 30(1) (2020), 23--66.

\bibitem{yws}
Y. Yu, T. Wang, and R. J. Samworth,
A useful variant of the Davis-Kahan theorem for statisticians, Biometrika 102 (2015), 315--323.

\bibitem{zuazua21}
C. Zhang and E. Zuazua,
A quantitative analysis of Koopman operator methods for system
identification and predictions,
Comptes Rendus. M\'ecanique, Online first (2023), 1--31.

\end{thebibliography}
\end{document}